\documentclass[10pt,french,english,a4paper]{article}
\usepackage[T1]{fontenc}
\usepackage{amsmath}
\usepackage[dvips]{graphicx}
\usepackage[symbol]{footmisc}

\usepackage{dsfont}

\usepackage{amsmath,amsfonts}

\usepackage{shadow}

\usepackage{latexsym}

\usepackage{latexsym}

\usepackage{textcomp}

\usepackage{indentfirst}

\usepackage{mathrsfs}

\usepackage{geometry}

\allowdisplaybreaks

\let\tilde\widetilde
\def\limnto{\mathrel{\mathop{\longrightarrow\kern 0pt}\limits_{n\to\infty}}}

\def\1{\mbox{1\hspace{-.25em}I}}

\let\tilde\widetilde

\newcommand{\theoremname}{Theorem}
\newcommand{\definitionname}{Definition}
\newcommand{\propositionname}{Proposition}
\newcommand{\lemmaname}{Lemma}
\newcommand{\corollaryname}{Corollary}
\newcommand{\propertyname}{Property}
\newcommand{\exercisename}{Exercise}
\newcommand{\remarkname}{Remark}
\newcommand{\recallname}{Recall}
\newcommand{\notationname}{Notation}

\newtheorem{theorem}{\theoremname}[section]
\newtheorem{lemma}[theorem]{\lemmaname}

\newtheorem{definition}[theorem]{\definitionname}

\newtheorem{remark}{\remarkname}

\newtheorem{proposition}[theorem]{\propositionname}

\usepackage{babel}
\addto\extrasfrench{
\providecommand{\fg}{\ifdim\lastskip>\z@\unskip\fi~\frqq}}

\begin{document}
\title{\Large{Another method of viscosity solutions of integro-differential partial equation by concavity}}
\date{\today}
\selectlanguage{english}

\author{\textsc{L.\ SYLLA\footnotemark[2]}}


\footnotetext[2]{Universit\'e Gaston Berger, LERSTAD, CEAMITIC, e-mail: sylla.lamine@ugb.edu.sn}
\maketitle

\begin{abstract}
In this paper we consider the problem of viscosity solution of integro-partial  differential equation(IPDE in short) via the solution of backward stochastic differential  equations(BSDE in short) with jumps where L\'evy's measure is not necessarily finite. We mainly use the concavity of the generator at the level of its second variable to establish the existence and uniqueness of the solution with non local terms.\\ 
\end{abstract}
\bigskip{}


\bigskip{}


\newpage
\section{Introduction}
We consider the following system of integro-partial differential equation, which is a function of $(t,x)$: $\forall i\in\{1,\ldots,m\}$,
\begin{equation}\label{eq1}
\left
\{\begin{array}{ll}
-\partial_{t}u^{i}(t,x)-b(t,x)^{\top}\mathrm{D}_{x}u^{i}(t,x)-\frac{1}{2}\mathrm{Tr}(\sigma\sigma^{\top}(t,x)\mathrm{D}^{2}_{xx}u^{i}(t,x))-\mathrm{K}_{i}u^{i}(t,x)\\
-\mathit{h}^{(i)}(t,x,u^{i}(t,x),(\sigma^{\top}\mathrm{D}_{x}u^{i})(t,x),\mathrm{B}_{i}u^{i}(t,x))=0,\quad (t,x)\in\left[ 0,T\right] \times\mathbb{R}^{k};\\
u^{i}(T,x)=g^{i}(x);
\end{array}
\right.
\end{equation}
where the operators $\mathrm{B}_{i}$ and $\mathrm{K}_{i}$ are defined as follows:
\begin{eqnarray}
\mathrm{B}_{i}u^{i}(t,x) & = & \displaystyle\int_{\mathrm{E}}\gamma^{i}(t,x,e)(u^{i}(t,x+\beta(t,x,e))-u^{i}(t,x))\lambda(\mathrm{d} e);\label{2.2}\\
\mathrm{K}_{i}u^{i}(t,x) & = & \displaystyle\int_{\mathrm{E}}(u^{i}(t,x+\beta(t,x,e))-u^{i}(t,x)-\beta(t,x,e)^{\top}\mathrm{D}_{x}u^{i}(t,x))\lambda(de).\nonumber
\end{eqnarray}
The resolution of (\ref{eq1}) is in connection with the following system of backward stochastic differential equations with jumps :

\begin{equation}\label{eq2}
\left
\{\begin{array}{ll}
dY^{i;t,x}_{s}=-f^{(i)}(s,X^{t,x}_{s},Y^{i;t,x}_{s},Z^{i;t,x}_{s},U^{i;t,x}_{s})ds+Z^{i;t,x}_{s}\mathrm{d}
\mathrm{B}_{s}+\displaystyle\int_{\mathrm{E}}\mathrm{U}^{i;t,x}
_{s}(e)\tilde{\mu}(\mathrm{d}s,\mathrm{d}e),\quad s\leq T;\\
Y^{i;t,x}_{T}=g^{i}(X^{t,x}_{T}); 
\end{array}
\right.
\end{equation}
and\\
the following standard stochastic differential equation of diffusion-jump type:
\begin{equation}\label{2.4}
X^{t,x}_{s}=x+\displaystyle\int^{s}_{t}b(r,X^{t,x}_{r})\, \mathrm{d}r+\displaystyle\int^{s}_{t}\sigma(r,X^{t,x}_{r})\, \mathrm{d}B_{r}+\displaystyle\int^{s}_{t}
\displaystyle\int_{E}\beta(r,X^{t,x}_{r-},e)\tilde{\mu}(dr,de),
\end{equation}
for $s\in[t,T]$ and $X^{t,x}_{s}=x$ if $s\leq t$.\\

Several authors have studied (\ref{eq1}) by various and varied methods.
Among others we can quote Barles and al. \cite{bar}, who use the theorem of comparison by supposition the monotony of the generator; Ouknine and Hamad\`ene \cite{hamaOuk} use the penalization method.\\
But recently Hamad\`ene does a relaxation on the hypothesis of monotony on the generator to introduce a new class of functions(see \cite{hama} page 216) for the resolution of (\ref{eq1}).\\
In this work we propose to solve (\ref{eq1}) by relaxing the monotonicity of the generator and the class of fonctions introduced in \cite{hama} and assuming that $\lambda=\infty$ and the concavity of the generator at the level of its second variable. We recall that this technic was used in \cite{fan} for the resolution of BSDE.\\
Our paper is organized as follows: in the next section we give the notations and the assumptions; in section $3$ we recall a number of existing results; in section $4$ we build estimates and properties for a good resolution of our problem;  section $5$ is reserved to give our main result.\\
And at the end, classical definition of the concept of viscosity solution is put in appendix.

\section{Notations and assumptions}
Let $\left(\Omega,\mathcal{F},(\mathcal{F}_{t})_{t\leq T},\mathbb{P}\right)$ be a stochastic basis such that $\mathcal{F}_{0}$ contains all $\mathbb{P}-$null sets of $\mathcal{F}$, and $\mathcal{F}_{t}=\mathcal{F}_{t+}:=\bigcap_{\epsilon>0}
\mathcal{F}_{t+\epsilon},~t\geq 0$, and we suppose that the filtration is generated by the two mutually independents processes:\\
(i) $B:=(B_{t})_{t\geq 0}$ a $d$-dimensional Brownian motion and,\\
(ii) a Poisson random measure $\mu$ on $\mathbb{R}^{+}\times\mathrm{E}$ where $\mathrm{E}:=\mathbb{R}^{\ell}-\{0\}$ is equipped with its Borel field $\mathcal{E}$ $(\ell\geq 1)$. The compensator $\nu(\mathrm{d}t,\mathrm{d}e)=\mathrm{d}t\lambda(\mathrm{d}e)$ is such that $\{\tilde{\mu}(\left[0,t\right]\times A)=(\mu-\lambda)(\left[ 0,t\right]\times A)\}_{t\geq 0}$ is a martingale for all $A\in\mathcal{E}$ satisfying $\lambda(A)<\infty$. We also assume that $\lambda$ is a $\sigma$-finite measure on $(E,\mathcal{E})$, integrates the function $(1\wedge\mid e\mid ^{2})$ and $\lambda(E)=\infty$.\\
Let's now introduce the following spaces:\\ 
(iii) $\mathcal{P}~(resp.~\mathbf{P})$ the field on $\left[0,T\right]\times \Omega$ of $\mathcal{F}_{t\leq T}$-progressively measurable (resp. predictable) sets.\\
(iv) For $\kappa\geq 1$, $\mathbb{L}^{2}_{\kappa}(\lambda)$ the space of Borel measurable functions $\varphi:=(\varphi(e))_{e\in E}$ from $E$ into $\mathbb{R}^{\kappa}$ such that 
$\|\varphi\|^{2}_{\mathbb{L}^{2}_{\kappa}(\lambda)}=\displaystyle\int_{E}\left|\varphi(e)\right|^{2}_{\kappa}\lambda(\mathrm{d}e)<\infty$; $\mathbb{L}^{2}_{1}(\lambda)$ will be simply denoted by $\mathbb{L}^{2}(\lambda)$;\\
(v) $\mathcal{S}^{2}(\mathbb{R}^{\kappa})$ the space of RCLL (for right continuous with left limits) $\mathcal{P}$-measurable and $\mathbb{R}^{\kappa}$-valued processes such that $\mathbb{E}[\sup_{s\leq T} \left|Y_{s}\right|^{2}]<\infty$;\\
(vi) $\mathbb{H}^{2}(\mathbb{R}^{\kappa\times d})$ the space of processes $Z:=(Z_{s})_{s\leq T}$ which are $\mathcal{P}$-measurable, $\mathbb{R}^{\kappa\times d}$-valued and satisfying $\mathbb{E}\left[\displaystyle\int^{T}_{0}\left|Z_{s}\right|^{2}\, \mathrm{d} s\right]<\infty$;\\
(vii) $\mathbb{H}^{2}(\mathbb{L}^{2}_{\kappa}(\lambda))$ the space of processes $U:=(U_{s})_{s\leq T}$ which are $\mathbf{P}$-measurable, $\mathbb{L}^{2}_{\kappa}(\lambda)$-valued and satisfying $\mathbb{E}\left[\displaystyle\int^{T}_{0}\|U_{s}(\omega)\|^{2}_{\mathbb{L}^{2}_{\kappa}(\lambda)}\, \mathrm{d} s\right]<\infty$;\\
(viii) $\Pi_{g}$ the set of deterministics functions\\ $\varpi:~(t,x)\in [0,T]\times \mathbb{R}^{\kappa}\mapsto\varpi(t,x)\in\mathbb{R}$ of polynomial growth, i.e., for which there exists two non-negative constants $C$ and $p$ such that for any $(t,x)\in [0,T]\times \mathbb{R}^{\kappa}$,
$$\left|\varpi(t,x)\right|\leq C(1+\left|x\right|^{p}).$$
The subspace of $\Pi_{g}$ of continuous functions will be denoted by $\Pi^{c}_{g}$;\\
(ix) $\mathcal{M}$ the class of functions which satisfy the $p$-order Mao condition in $x$ i.e. If $f\in\mathcal{M}$ then there  exists a nondecreasing, continuous and concave function $\rho(\cdot):\mathbb{R}^{+}\mapsto\mathbb{R}^{+}$ with $\rho(0)=0$, $\rho(u)>0$, for $u>0$ and $\displaystyle\int_{0^{+}}\frac{du}{\rho(u)}=+\infty$, such that $d\mathbb{P}\times dt-$a.e., $\forall x, x^{'}\in\mathbb{R}^{k}\textrm{ and }\forall p\geq 2$,
$$|f(t,x,y,z,q)-f(t,x^{'},y,z,q)|\leq\rho^{\frac{1}{p}}(|x-x^{'}|^{p});$$
(x) For any process $\theta:=(\theta_{s})_{s\leq T}$ and $t\in(0,T],~\theta_{t-}=\lim_{s\nearrow t}\theta_{s}$ and\\
$$\Delta_{t}\theta=\theta_{t}-\theta_{t-}.$$
Now let $b$ and $\sigma$  be the following functions:
$$b:(t,x)\in [0,T]\times \mathbb{R}^{k}\mapsto b(t,x)\in\mathbb{R}^{k};$$
$$\sigma:(t,x)\in [0,T]\times \mathbb{R}^{k}\mapsto\sigma(t,x)\in\mathbb{R}^{k\times d}.$$
\begin{equation}
\textrm{We assume that } b \textrm{ and } \sigma \textrm{ belong to } \mathcal{M}\label{2.5}
\end{equation}

Let $\beta:(t,x,e)\in [0,T]\times \mathbb{R}^{k}\times E\mapsto \beta(t,x,e)\in\mathbb{R}^{k}$ be a measurable function such that for some real constant $C$, and for all $e\in E$,
\begin{eqnarray}\label{2.7}
& (i) & \left|\beta(t,x,e)\right|  \leq  C(1\wedge\left|e\right|); \\ \nonumber  
& (ii) & \beta \textrm{ belongs to } \mathcal{M} ;\\
& (iii) & \textrm{the mapping}~(t,x)\in[0,T]\times \mathbb{R}^{k}\mapsto \beta(t,x,e)\in\mathbb{R}^{k}~\textrm{is continuous for any}~\mathrm{e}\in\mathrm{E}\nonumber.
\end{eqnarray}
The functions $(g^{i})_{i=1,m}$ and $(h^{(i)})_{i=1,m}$ be two functions defined as follows: for $i=1,\ldots,m$,
\begin{eqnarray}
g^{i}:\mathbb{R}^{k} & \longrightarrow & \mathbb{R}^{m}\nonumber \\
{}{}x & \longmapsto & g^{i}(x)\nonumber
\end{eqnarray}
and
\begin{eqnarray}
h^{(i)}:[0,T]\times\mathbb{R}^{k+m+d+1} & \longrightarrow & \mathbb{R}\nonumber \\
{}{}(t,x,y,z,q) & \longmapsto & h^{(i)}(t,x,y,z,q).\nonumber
\end{eqnarray}
Moreover we assume they satisfy:\\
(\textbf{H1}): For any $i\in\left\lbrace 1,\ldots,m\right\rbrace$, the function $g^{i}$ belongs to $\mathcal{M}$.\\
\\
(\textbf{H2}): For any $i\in\left\lbrace 1,\ldots,m\right\rbrace$,
\begin{eqnarray} 
& (i) &~\textrm{the function}~h^{(i)}~\textrm{is Lipschitz in}~ (y,z,q)~\textrm{uniformly in}~(t,x),~\textrm{i.e., there exists a real constant}\nonumber\\
& {}{} & \textrm{C such that for any}~ (t,x)\in[0,T]\times\mathbb{R}^{k}, (y,z,q)~\textrm{and}~(y',z',q')~\textrm{elements of}~\mathbb{R}^{m+d+1},\nonumber\\
& {}{} &\left|h^{(i)}(t,x,y,z,q)-h^{(i)}(t,x,y',z',q')\right|\leq C(\left|y-y'\right|+\left|z-z'\right|+\left|q-q'\right|);\label{2.11}\\
& (ii) & ~\textrm{the}~(t,x)\mapsto h^{(i)}(t,x,y,z,q), ~\textrm{for fixed}~ (y,z,q)\in\mathbb{R}^{m+d+1},~\textrm{belongs uniformly to}~\mathcal{M},\nonumber
\end{eqnarray}
Next let $\gamma^{i},~i=1,\ldots,m$ be Borel measurable functions defined from $[0,T]\times\mathbb{R}^{k}\times E$ into $\mathbb{R}$ and satisfying:
\begin{eqnarray}\label{2.12} 
& (i) &\left|\gamma^{i}(t,x,e)\right|\leq C(1\wedge\left|e\right|);\nonumber\\
& (ii) & \gamma^{i} \textrm{ belongs to }\mathcal{M};\\
& (iii) & \textrm{the mapping}~t\in[0,T]\mapsto \gamma^{i}(t,x,e)\in\mathbb{R}~ \textrm{is continuous for any}~(x,e)\in\mathbb{R}^{k}\times E.\nonumber
\end{eqnarray}
Finally we introduce the following functions $(f^{(i)})_{i=1,m}$ defined by:
\begin{equation}\label{2.13}
\forall (t,x,y,z,\zeta)\in[0,T]\times\mathbb{R}^{k+m+d}\times \mathbb{L}^{2}(\lambda),~f^{(i)}(t,x,y,z,\zeta):=h^{(i)}\left(t,x,y,z,\displaystyle\int_{E}\gamma^{i}(t,x,e)\zeta(e)\lambda(de)\right).
\end{equation}
The functions $(f^{(i)})_{i=1,m}$, enjoy the two following properties:

\begin{eqnarray}\label{2.15} 
& (a) &~\textrm{The function}~f^{(i)}~\textrm{is Lipschitz in}~ (y,z,\zeta)~\textrm{uniformly in}~(t,x),~\textrm{i.e., there exists a real constant}\nonumber\\
& {}{} & \textrm{C such that}\nonumber\\
& {}{} &\left|f^{(i)}(t,x,y,z,\zeta)-f^{(i)}(t,x,y',z',\zeta')\right|\leq C(\left|y-y'\right|+\left|z-z'\right|+\|\zeta-\zeta'\|_{\mathbb{L}^{2}(\lambda)});\\
\nonumber
& {}{} & \textrm{since}~h^{(i)}~\textrm{is uniformly Lipschitz in}~(y,z,q)~\textrm{and}~ \gamma^{i}~\textrm{verifies (\ref{2.12})-(i)};\nonumber\\
& (b) &~\textrm{The function}~(t,x)  \in[0,T]\times\mathbb{R}^{k}\mapsto f^{(i)}(t,x,0,0,0)~\textrm{belongs}~to~\Pi^{c}_{g};\nonumber\\
\nonumber
& {}{} & \textrm{and then}~ \mathbb{E}\left[\displaystyle\int^{T}_{0}\left|f^{(i)}(r,X^{t,x}_{r},0,0,0)\right|^{2}\,dr \right]<\infty.
\end{eqnarray}

\section{Some results in backward stochastic differential equation with jumps}
\subsection{A class of diffusion processes with jumps}
Let $(t,x)\in[0,T]\times\mathbb{R}^{d}$ and $(X^{t,x}_{s})_{s\leq T}$ be the stochastic process solution of (\ref{2.4}).
Under assumptions (\ref{2.5})-(\ref{2.7}) the solution of Equation (\ref{2.4}) exists and is unique (see \cite{fuj} for more details).
We state some properties of the process $\{(X^{t,x}_{s}),~s\in[0,T]\}$ which can found in \cite{fuj}. 
\begin{proposition}\label{pro3.1}
For each $t\geq 0$, there exists a version of $\{(X^{t,x}_{s}),~s\in[t,T]\}$ such that $s\rightarrow X^{t}_{s}$ is a $C^{2}(\mathbb{R}^{d})$-valued rcll process. Moreover it satisfies the following estimates: $\forall p\geq 2,~x,x'\in\mathbb{R}^{d}$ and $s\geq t$,
\begin{eqnarray}
\mathbb{E}[\sup_{t\leq r\leq s} \left|X^{t,x}_{r}-x\right|^{p}] & \leq &  M_{p}(s-t)(1+\left|x\right|^{p});\nonumber\\
\mathbb{E}[\sup_{t\leq r\leq s} \left|X^{t,x}_{r}-X^{t,x'}_{r}-(x-x')^{p}\right|^{p}] & \leq &  M_{p}(s-t)(\left|x-x'\right|^{p});\label{3.16}
\end{eqnarray}
for some constant $M_{p}$.
\end{proposition}
\subsection{Existence and uniqueness for BSDE with jumps}
Let $(t,x)\in[0,T]\times\mathbb{R}^{d}$  and we consider the following m-dimensional BSDE with jumps:
\begin{equation}\label{3.17}
\left
\{\begin{array}{ll}
(i)~\vec{Y}^{t,x}:=(Y^{i,t,x})_{i=1,m}\in\mathcal{S}^{2}(\mathbb{R}^{m}),~Z^{t,x}:=(Z^{i,t,x})_{i=1,m}\in\mathbb{H}^{2}(\mathbb{R}^{m\times d}),~ U^{t,x}:=(U^{i,t,x})_{i=1,m}\in\mathbb{H}^{2}(\mathbb{L}^{2}_{m}(\lambda));\\
(ii)~dY^{i;t,x}_{s}=-f^{(i)}(s,X^{t,x}_{s},Y^{i;t,x}_{s},Z^{i;t,x}_{s},U^{i;t,x}_{s})ds+Z^{i;t,x}_{s}\mathrm{d}
\mathrm{B}_{s}+\displaystyle\int_{\mathrm{E}}\mathrm{U}^{i;t,x}
_{s}(e)\tilde{\mu}(\mathrm{d}s,\mathrm{d}e),\quad s\leq T;\\ 
(iii)~Y^{i;t,x}_{T}=g^{i}(X^{t,x}_{T});
\end{array}
\right.
\end{equation}
where for any $i\in\{1,\ldots,m\}$, $Y^{i;t,x}_{s}$ is the ith row of $Y^{t,x}_{s}$, $Z^{i;t,x}_{s}$ is the ith component of $Z^{t,x}_{s}$ and $U^{i;t,x}_{s}$ is the ith component of $U^{t,x}_{s}$.\\
\begin{proposition}\label{prop3.2}
Assume that assumptions $(\mathbf{H1})\textrm{ and }(\mathbf{H2})$ hold. Then for any $(t,x)\in[0,T]\times\mathbb{R}^{d}$, the BSDE (\ref{3.17}) has an unique solution $(\vec{Y}^{t,x},Z^{t,x},U^{t,x})$.
\end{proposition}
For proof of this proposition we can see \cite{bar}.
\subsection{Viscosity solutions of integro-differential partial equation}
\begin{proposition}(see, \cite{bar})\label{prop3.3}
Assume that (\textbf{H1}), (\textbf{H2}), are fulfilled. Then there exists deterministic continuous functions $(u^{i}(t,x))_{i=1,m}$ which belong to $\Pi_{g}$ such that for any $(t,x)\in[0,T]\times\mathbb{R}^{k}$, the solution of the BSDE (\ref{3.17}) verifies:
\begin{equation}\label{3.18}
\forall i\in\{1,\ldots,m\},~\forall s\in[t,T],~Y^{i;t,x}_{s}=u^{i}(s,X^{t,x}_{s}).
\end{equation} 
Moreover if for any $i\in\{1,\ldots,m\}$,
\begin{eqnarray*} 
& (i) & \gamma^{i}\geq 0;\\
& (ii) & \textrm{for any fixed}~(t,x,\vec{y},z)\in[0,T]\times \mathbb{R}^{k+m+d},~\textrm{the mapping}\\ 
& {}{} & (q\in\mathbb{R})\longmapsto h^{(i)}(t,x,\vec{y},z,q)\in\mathbb{R}~\textrm{is non-decreasing}. 
\end{eqnarray*}
\end{proposition}
The function $(u^{i})_{i=1,m}$ is a continuous viscosity solution (in Barles and al. 's sense, see Definition \ref{def5.3} in the Appendix) of (\ref{eq1}). The solution $(u^{i})_{i=1,m}$ of (\ref{eq1}) is unique in the class $\Pi^{c}_{g}$.
\begin{remark}(see, \cite{bar})
Under the assumptions (\textbf{H1}), (\textbf{H2}), there exists a unique viscosity solution of (\ref{eq1}) in the class of functions satisfying
\begin{equation}
\lim\limits_{\left|x\right| \to +\infty}\left|u(t,x)\right| e^{-\tilde{A}[\log(\left|x\right|)]^{2}}=0
\end{equation} 
uniformly for $t\in[0,T]$, for some $\tilde{A}>0$.
\end{remark}
\section{Estimates and properties}
In this section, we will establish a priori estimates concerning solutions of BSDE (\ref{3.17}), which will play an important role in the proof of our main results.  
\begin{lemma}(see, \cite{hama})\label{lem4.1}
Under assumption (\textbf{H1}), (\textbf{H2}), for any $p\geq 2$ there exists two non-negative constants $C$ and $\delta$ such that,
\begin{equation}\label{4.22}
\mathbb{E}\left[\left\lbrace \displaystyle\int^{T}_{0} ds \left(\displaystyle\int_{E}\left|U^{t,x}_{s}(e)\right|^{2}\lambda(de)\right)\right\rbrace^{\frac{p}{2}}\right]=\mathbb{E}\left[\left\lbrace \displaystyle\int^{T}_{0} ds\|U^{t,x}_{s}\|^{2}_{\mathbb{L}^{2}_{m}(\lambda)}\right\rbrace^{\frac{p}{2}}\right]\leq C\left(1+\left|x\right|^{\delta}\right). 
\end{equation} 
\end{lemma}
\begin{proposition}\label{prop4.2}
For any $i=1,\ldots,m$, there exist $C\geq 0$, $\kappa>0$ such that, $\forall$ $x$ and $x'$ elements of $\mathbb{R}^{k}$ $$|u^{i}(t,x)-u^{i}(t,x')|^{2}\leq \rho\left(M_{2}\left|x-x'\right|^{2}(1+\left|x-x'\right|^{2})\right)\left[C(1+\left|x\right|^{\kappa})\right]$$
\end{proposition}
\begin{proof}
Let $x$ and $x'$ be elements of $\mathbb{R}^{k}$. Let $(\vec{Y}^{t,x},Z^{t,x},U^{t,x})$ $(\textrm{resp. }(\vec{Y}^{t,x'},Z^{t,x'},U^{t,x'}))$ be the solution of the BSDE with jumps (\ref{3.17}) associated with ($f(s,X^{t,x}_{s},y,\eta,\zeta),g(X^{t,x}_{T})$)\\
$(\textrm{resp. }f(s,X^{t,x'}_{s},y,\eta,\zeta),g(X^{t,x'}_{T}))$. Applying It\^o formula to $\left|\vec{Y}^{t,x}-\vec{Y}^{t,x'}\right|^{2}$ between $s$ and $T$, we have
\newpage
\begin{eqnarray}\label{4.31}
& {}{} &\left|\vec{Y}^{t,x}_{s}-\vec{Y}^{t,x'}_{s}\right|^{2}+\displaystyle\int^{T}_{s}\left|\Delta Z_{r}\right|^{2}\,dr+\sum_{s\leq r\leq T}(\Delta_{r}\vec{Y}^{t,x}_{r})^{2}\\
& {}{} &=\left|g(X^{t,x}_{T})-g(X^{t,x'}_{T})\right|^{2}+2\displaystyle\int^{T}_{s}
<\left(\vec{Y}^{t,x}_{s}-\vec{Y}^{t,x'}_{s}\right),\Delta f(r)>\,dr\nonumber\\
& {}{} &\quad\quad-2\displaystyle\int^{T}_{s}
\displaystyle\int_{\mathrm{E}}\left(\vec{Y}^{t,x}_{r}
-\vec{Y}^{t,x'}_{r}\right)\left(\Delta U_{r}(e)\right)\tilde{\mu}(\mathrm{d}r,\mathrm{d}e)-2\displaystyle\int^{T}_{s} \left(\vec{Y}^{t,x}_{r}-\vec{Y}^{t,x'}_{r}\right)\left(\Delta Z_{r}\right)\,dB_{r};\nonumber
\end{eqnarray}
and taking expectation we obtain: $\forall s\in[t,T]$,
\begin{eqnarray}\label{4.33}
& {}{} &\mathbb{E}\left[\left|\vec{Y}^{t,x}_{s}-\vec{Y}^{t,x'}_{s}\right|^{2}+\displaystyle\int^{T}_{s}\left|\Delta Z_{r}\right|^{2}\,dr+\displaystyle\int^{T}_{s}\|\Delta U_{r}\|^{2}_{\mathbb{L}^{2}(\lambda)}\,dr\right]\\
& {}{} &\leq\mathbb{E}\left[\left|g(X^{t,x}_{T})-g(X^{t,x'}_{T})\right|^{2}+2\displaystyle\int^{T}_{s}
<\left(\vec{Y}^{t,x}_{s}-\vec{Y}^{t,x'}_{s}\right),\Delta f(r)>\,dr\right]\nonumber
\end{eqnarray}
where the processes $\Delta X_{r}$, $\Delta Y_{r}$, $\Delta f(r)$, $\Delta Z_{r}$ and $\Delta U_{r}$  are defined as follows: $\forall r\in[t,T]$,\\
$\Delta f(r):=((\Delta f^{(i)}(r))_{i=1,m}=(f^{(i)}(r,X^{i;t,x}_{r},\vec{Y}^{t,x}_{r},Z^{i;t,x}_{r},U^{i;t,x}_{r})-f^{(i)}(r,X^{i;t,x'}_{r},\vec{Y}^{i;t,x'}_{r},Z^{i;t,x'}_{r},U^{i;t,x'}_{r}))_{i=1,m}$,
$\Delta X_{r}=X^{t,x}_{r}-X^{t,x'}_{r}$, $\Delta Y(r)=\vec{Y}^{t,x}_{r}-\vec{Y}^{t,x'}_{r}=(Y^{j;t,x}_{r}-Y^{j;t,x'}_{r})_{j=1,m}$,\\
$\Delta Z_{r}=Z^{t,x}_{r}-Z^{t,x'}_{r}$ and $\Delta U_{r}=U^{t,x}_{r}-U^{t,x'}_{r}$
($<\cdot,\cdot>$ is the usual scalar product on $\mathbb{R}^{m}$).
Now we will give an estimation of each three terms of the second member of inequality (\ref{4.33}).\\
$\bullet$ As for any $i\in\{1,\ldots,m\}$ $g^{i}$ belongs to $\mathcal{M}~(\textrm{ for }p=2)$; therefore\\
\begin{eqnarray}
\mathbb{E}\left[\left|g(X^{t,x}_{T})-g(X^{t,x'}_{T})\right|^{2}\right] & \leq & \mathbb{E}\left[\rho\left(\left|X^{t,x}_{T}-X^{t,x'}_{T}\right|^{2}\right)\right]\nonumber\\
& \leq & \rho\left(\mathbb{E}\left[\left|X^{t,x}_{T}-X^{t,x'}_{T}\right|^{2}\right]\right)\left(\textrm{by Jensen's inequality}\right)\nonumber\\
& \leq & \rho\left(\mathbb{E}\left[\left|X^{t,x}_{T}-X^{t,x'}_{T}-(x-x')^{2}+(x-x')^{2}\right|^{2}\right]\right)\nonumber
\end{eqnarray}
and by subsequently using the triangle inequality, the relation (\ref{3.16}) of proposition \ref{pro3.1} and the fact that
 $$(a+b){^p}\leq 2^{p-1}(a^{p}+b^{p}).$$
\begin{equation}\label{4.34}
\mathbb{E}\left[\left|g(X^{t,x}_{T})-g(X^{t,x'}_{T})\right|^{2}\right]\leq \rho\left(M_{2}\left|x-x'\right|^{2}(1+\left|x-x'\right|^{2})\right),
\end{equation} 
$\bullet$ To complete our estimation of (\ref{4.33}) we need to deal with $\mathbb{E}\left[2\displaystyle\int^{T}_{s}
<\left(\vec{Y}^{t,x}_{s}-\vec{Y}^{t,x'}_{s}\right),\Delta f(r)>\,dr\right].$
Taking into account the expression of $f^{(i)}$ given by (\ref{2.13}) we then split $\Delta f(r)$ in the follows way: for $r\leq T$,
$$\Delta f(r)=(\Delta f(r))_{i=1,m}=\Delta_{1}(r)+\Delta_{2}(r)+\Delta_{3}(r)+\Delta_{4}(r)=(\Delta^{i}_{1}(r)+\Delta^{i}_{2}(r)+\Delta^{i}_{3}(r)+\Delta^{i}_{4}(r))_{i=1,m},$$
\newpage
where  for any $i=1,\ldots,m$, 
\begin{eqnarray*}
\Delta^{i}_{1}(r) & = & h^{(i)}\left(r,X^{t,x}_{r},\vec{Y}^{t,x}_{r},Z^{i;t,x}_{r},\displaystyle\int_{E}\gamma^{i}(r,X^{t,x}_{r},e)U^{i;t,x}_{r}(e)\lambda(de)\right)\\
&{}{}& -h^{(i)}\left(r,X^{t,x'}_{r},\vec{Y}^{t,x}_{r},Z^{i;t,x}_{r},\displaystyle\int_{E}\gamma^{i}(r,X^{t,x}_{r},e)U^{i;t,x}_{r}(e)\lambda(de)\right);\\ 
\Delta^{i}_{2}(r) & = & h^{(i)}\left(r,X^{t,x'}_{r},\vec{Y}^{t,x}_{r},Z^{i;t,x}_{r},\displaystyle\int_{E}\gamma^{i}(r,X^{t,x}_{r},e)U^{i;t,x}_{r}(e)\lambda(de)\right)\\
&{}{}& -h^{(i)}\left(r,X^{t,x'}_{r},\vec{Y}^{t,x'}_{r},Z^{i;t,x}_{r},\displaystyle\int_{E}\gamma^{i}(r,X^{t,x}_{r},e)U^{i;t,x}_{r}(e)\lambda(de)\right);\\ 
\Delta^{i}_{3}(r) & = & h^{(i)}\left(r,X^{t,x'}_{r},\vec{Y}^{t,x'}_{r},Z^{i;t,x}_{r},\displaystyle\int_{E}\gamma^{i}(r,X^{t,x}_{r},e)U^{i;t,x}_{r}(e)\lambda(de)\right)\\
&{}{}& -h^{(i)}\left(r,X^{t,x'}_{r},\vec{Y}^{t,x'}_{r},Z^{i;t,x'}_{r},\displaystyle\int_{E}\gamma^{i}(r,X^{t,x}_{r},e)U^{i;t,x}_{r}(e)\lambda(de)\right);\\ 
\Delta^{i}_{4}(r) & = & h^{(i)}\left(r,X^{t,x'}_{r},\vec{Y}^{t,x'}_{r},Z^{i;t,x'}_{r},\displaystyle\int_{E}\gamma^{i}(r,X^{t,x}_{r},e)U^{i;t,x}_{r}(e)\lambda(de)\right)\\
&{}{}& -h^{(i)}\left(r,X^{t,x'}_{r},\vec{Y}^{t,x'}_{r},Z^{i;t,x'}_{r},\displaystyle\int_{E}\gamma^{i}(r,X^{t,x'}_{r},e)U^{i;t,x'}_{r}(e)\lambda(de)\right).
\end{eqnarray*}
By Cauchy-Schwartz inequality, the inequality $2ab\leq\epsilon a^{2}+\frac{1}{\epsilon}b^{2}$, (\textbf{H2})-(ii), the estimate (\ref{3.16}) and Jensen's inequality we have: 
\begin{eqnarray}\label{4.36}
\mathbb{E}\left[2\displaystyle\int^{T}_{s}
\scriptstyle<\Delta Y(r),\Delta_{1}(r)>\,dr\right] & \leq & \mathbb{E}\left[\frac{1}{\epsilon}\int^{T}_{s}\scriptstyle{|\Delta Y(r)|^{2}\,dr+\epsilon}\displaystyle\int^{T}_{s}\rho\left(|X^{t,x}_{r}-X^{t,x'}_{r}|^{2}\right)\,dr\right]\nonumber\\
& \leq & \mathbb{E}\left[\frac{1}{\epsilon}\int^{T}_{s}|\Delta Y(r)|^{2}\,dr\right]+\epsilon\rho\left(M_{2}\left|x-x'\right|^{2}(1+\left|x-x'\right|^{2})\right).
\end{eqnarray}
Besides since $h^{(i)}$ is Lipschitz w.r.t. $(y,z,q)$ then,
\begin{equation}\label{4.37}
\mathbb{E}\left[2\displaystyle\int^{T}_{s}<\Delta Y(r),\Delta_{2}(r)>\,dr\right]\leq 2C\mathbb{E}\left[\int^{T}_{s}|\Delta Y(r)|^{2}\,dr\right],
\end{equation}
and
\begin{equation}\label{4.38}
\mathbb{E}\left[2\displaystyle\int^{T}_{s}<\Delta Y(r),\Delta_{3}(r)>\,dr\right]\leq\mathbb{E}\left[\frac{1}{\epsilon}\int^{T}_{s}|\Delta Y(r)|^{2}\,dr+C^{2}\epsilon\int^{T}_{s}|\Delta Z(r)|^{2}\,dr\right].
\end{equation}
It remains to obtain a control of the last term. But for any $s\in[t,T]$ we have,
\begin{eqnarray}\label{4.39}
& {}{}& \mathbb{E}\left[2\displaystyle\int^{T}_{s}<\Delta Y(r),\Delta_{4}(r)>\,dr\right]\\
& \leq & 2C\mathbb{E}\left[\int^{T}_{s}|\Delta Y(r)|\,dr\times \left|\int_{E}\left(\gamma(r,X^{t,x}_{r},e)U^{t,x}_{r}(e)-\gamma(r,X^{t,x'}_{r},e)U^{t,x'}_{r}(e)\right)\,\lambda(de)\right|\right]\nonumber.
\end{eqnarray}
Next by splitting the crossing terms as follows
$\gamma(r,X^{t,x}_{r},e)U^{t,x}_{r}(e)-\gamma(r,X^{t,x'}_{r},e)U^{t,x'}_{r}(e)=\Delta U_{s}(e)\gamma(s,X^{t,x}_{s},e)+U^{t,x'}_{s}\left(\gamma(s,X^{t,x}_{s},e)-\gamma(s,X^{t,x'}_{s},e)\right)$\\
and setting $\Delta \gamma_{s}(e):=\left(\gamma(s,X^{t,x}_{s},e)-\gamma(s,X^{t,x'}_{s},e)\right)$,\\
we obtain,
\begin{eqnarray}\label{4.40}
\mathbb{E}\left[2\displaystyle\int^{T}_{s}<\Delta Y(r),\Delta_{4}(r)>\,dr\right]& \leq & 2C\mathbb{E}\left[\int^{T}_{s}|\scriptstyle\Delta Y(r)|\times\left(\displaystyle\int_{E}\scriptstyle(|U^{t,x'}_{r}(e)\Delta\gamma_{r}(e)|+|\Delta U_{r}(e)\gamma(r,X^{t,x}_{r},e)|)\,\lambda(de)\right)\,dr\right]\nonumber\\
& \leq & \frac{2}{\epsilon}\mathbb{E}\left[\int^{T}_{s}|\Delta Y(r)|^{2}\,dr\right]+C^{2}\epsilon\mathbb{E}\left[\int^{T}_{s}\left(\int_{E}(|U^{t,x'}_{r}(e)\Delta\gamma_{r}(e)|\lambda(de)\right)^{2}\,dr\right]\nonumber\\
& {}{} &+C^{2}\epsilon\mathbb{E}\left[\int^{T}_{s}\left(\int_{E}(|\Delta U_{r}(e)\gamma(r,X^{t,x}_{r},e)|\lambda(de)\right)^{2}\,dr\right].
\end{eqnarray}
By Cauchy-Schwartz inequality, (\ref{2.12})-(ii), Jensen's inequality and (\ref{3.16}), and the result of Lemma \ref{lem4.1} it holds: 
\begin{eqnarray}\label{4.41}
\mathbb{E}\left[\int^{T}_{s}\left(\int_{E}(|U^{t,x'}_{r}(e)\Delta\gamma_{r}(e)|\lambda(de)\right)^{2}\,dr\right] & \leq & \mathbb{E}\left[\int^{T}_{s}\,dr\left(\int_{E}|U^{t,x'}_{r}(e)|^{2}\lambda(de)\right)\left(\int_{E}|\Delta\gamma_{r}(e)|^{2}\lambda(de)\right)\right]\nonumber\\
&\leq & \rho\left(M_{2}\left|x-x'\right|^{2}(1+\left|x-x'\right|^{2})\right)\times \mathbb{E}\left[\int^{T}_{s}\,dr\left(\int_{E}|U^{t,x'}_{r}(e)|^{2}\lambda(de)\right)\right]\nonumber\\
& \leq & C\rho\left(M_{2}\left|x-x'\right|^{2}(1+\left|x-x'\right|^{2})\right)(1+\left|x\right|^{\kappa}).
\end{eqnarray}
On the other hand using once more Cauchy-Schwartz inequality and (\ref{2.12})-(i) we get
\begin{eqnarray}\label{4.42}
\mathbb{E}\left[\int^{T}_{s}\left(\int_{E}(\scriptstyle|\Delta U_{r}(e)\gamma(r,X^{t,x}_{r},e)|\lambda(de)\right)^{2}\,dr\right] & \leq & \mathbb{E}\left[\int^{T}_{s}\,dr\left(\int_{E}(\scriptstyle|\Delta U_{r}(e)|^{2}\lambda(de)\right)\left(\int_{E}|\gamma(r,X^{t,x}_{r},e)|^{2}\lambda(de)\right)\right]\nonumber\\
& \leq & C\mathbb{E}\left[\int^{T}_{s}\,dr\left(\int_{E}(|\Delta U_{r}(e)|^{2}\lambda(de)\right)\right].
\end{eqnarray}
From (\ref{4.36}) to (\ref{4.42}) it follows that:
\begin{eqnarray*}\label{4.43}
& {}{} &\mathbb{E}\left[\left|\vec{Y}^{t,x}_{s}-\vec{Y}^{t,x'}_{s}\right|^{2}+\displaystyle\int^{T}_{s}\left|\Delta Z_{r}\right|^{2}\,dr+\displaystyle\int^{T}_{s}\|\Delta U_{r}\|^{2}_{\mathbb{L}^{2}(\lambda)}\,dr\right]\\
& {}{} &\leq\mathbb{E}\left[\left|g(X^{t,x}_{T})-g(X^{t,x'}_{T})\right|^{2}+2\displaystyle\int^{T}_{s}
<\left(\vec{Y}^{t,x}_{s}-\vec{Y}^{t,x'}_{s}\right),\Delta f(r)>\,dr\right]\\
& \leq & \rho\left(M_{2}\left|x-x'\right|^{2}(1+\left|x-x'\right|^{2})\right)\left[C(1+\left|x\right|^{\kappa})+1+\epsilon+\epsilon C^{3}\right]+\left(\frac{4}{\epsilon}+2C\right)\mathbb{E}\left[\int^{T}_{s}|\Delta Y(r)|^{2}\,dr\right]\\
&{}{}&+C^{2}\epsilon\mathbb{E}\left[\int^{T}_{s}|\Delta Z(r)|^{2}\,dr\right]+C^{3}\epsilon\mathbb{E}\left[\int^{T}_{s}\,dr\left(\int_{E}(|\Delta U_{r}(e)|^{2}\lambda(de)\right)\right].
\end{eqnarray*}
By choosing $\epsilon$ such that $\{\epsilon+\frac{4}{\epsilon}+\epsilon(2C^{3}+C^{2})+2C<1\}$ we deduce the existence of a constant $C\geq 0$ such that for any $s\in[t,T]$,\\
$\mathbb{E}\left[|\Delta Y(s)|^{2}\right]\leq \rho\left(M_{2}\left|x-x'\right|^{2}(1+\left|x-x'\right|^{2})\right)\left[C(1+\left|x\right|^{\kappa})\right]+\mathbb{E}\left[\displaystyle\int^{T}_{s}|\Delta Y(r)|^{2}\,dr\right]$\\
and by Gronwall lemma this implies that for any $s\in[t,T]$,\\
$$\mathbb{E}\left[|\Delta Y(s)|^{2}\right]\leq \rho\left(M_{2}\left|x-x'\right|^{2}(1+\left|x-x'\right|^{2})\right)\left[C(1+\left|x\right|^{\kappa})\right].$$
Finally in taking $s=t$ and considering (\ref{3.18}) we obtain the desired result.\\
\\
\end{proof}
Now we start a point which giving difference of definition viscosity solution between \cite{bar} and \cite{hama}.\\
It should also be noted that in this part will appear our first contribution after of course the first corresponding to the proposition \ref{prop4.2}.
It will be mainly about the use of the $\mathcal{M}$ class and the Bihari inequality as in \cite{fan}.  
\begin{proposition}\label{pro4.3}
For any $i=1,\ldots,m$, $(t,x)\in[0,T]\times\mathbb{R}^{k}$,
\begin{equation}
U^{i;t,x}_{s}(e)=u^{i}(s,X^{t,x}_{s-}+\beta(s,X^{t,x}_{s-},e))-u^{i}(s,X^{t,x}_{s-}),~~d\mathbb{P}\otimes ds\otimes d\lambda-\text{a.e. on}~\Omega\times[t,T]\times E.
\end{equation}

\end{proposition}
\begin{proof}
\textbf{Step 1: Truncation of the L\'evy measure}\\
For any $k\geq 1$, let us first introduce a new Poisson random measure $\mu_{k}$ (obtained from the truncation of $\mu$) and its associated compensator $\nu_{k}$ as follows:
$$\mu_{k}(ds,de)=1_{\{|e|\geq\frac{1}{k}\}}\mu(ds,de)~~\text{and }\nu_{k}(ds,de)=1_{\{|e|\geq\frac{1}{k}\}}\nu(ds,de).$$ 
Which means that, as usual, $\tilde{\mu_{k}}(ds,de):=(\mu_{k}-\nu_{k})(ds,de)$, is the associated random martingale measure.\\
The main point to notice is that 
\begin{eqnarray}
\lambda_{k}(E)=\displaystyle\int_{E}\,\lambda_{k}(de)& = &\displaystyle\int_{E}1_{\{|e|\geq\frac{1}{k}\}}\,\lambda(de)\nonumber\\
{}{}&=&\displaystyle\int_{\{|e|\geq\frac{1}{k}\}}\,\lambda(de)\nonumber\\
{}{}&=&\lambda(\{|e|\geq\frac{1}{k}\})<\infty.
\end{eqnarray}
As in \cite{hama}, let us introduce the process $^{k}X^{t,x}$ solving the following standard SDE of jump-diffusion type:
\begin{eqnarray}
& {}{} & ^{k}X^{t,x}_{s}=x+\displaystyle\int^{s}_{t}b(r,^{k}X^{t,x}_{r})\,dr+\displaystyle\int^{s}_{t}\sigma(r,^{k}X^{t,x}_{r})\,dB_{r}\nonumber\\
& {}{} &\qquad\qquad+\displaystyle\int^{s}_{t}\displaystyle\int_{\mathrm{E}}\beta(r,^{k}X^{t,x}_{r-},e)\tilde{\mu}_{k}\,(dr,de),~~~t\leq s\leq T;~^{k}X^{t,x}_{r}=x~\text{if }s\leq t.\nonumber\\
\end{eqnarray}
 Note that thanks to the assumptions on $b$, $\sigma$, $\beta$ the process $^{k}X^{t,x}$ exists and is unique. Moreover it satisfies the same estimates as in (\ref{3.16}) since $\lambda_{k}$ is just a truncation at the origin of $\lambda$ which integrates $(1\wedge|e|^{2})_{e\in E}$.\\
On the other hand let us consider the following Markovian BSDE with jumps
 \begin{equation}\label{4.46}
\left
\{\begin{array}{ll}
(i)~\mathbb{E}\left[\sup_{s\leq T}\left|^{k}Y^{t,x}_{s}\right|^{2}+\displaystyle\int^{T}_{s}\left|^{k}Z^{t,x}_{r}\right|^{2}\,dr+\displaystyle\int^{T}_{s}\|^{k}U^{t,x}_{r}\|^{2}_{\mathbb{L}^{2}(\lambda_{k})}\,dr\right]<\infty\\
(ii)~^{k}{Y}^{t,x}:=(^{k}Y^{i,t,x})_{i=1,m}\in\mathcal{S}^{2}(\mathbb{R}^{m}),~^{k}Z^{t,x}:=(^{k}Z^{i,t,x})_{i=1,m}\in\mathbb{H}^{2}(\mathbb{R}^{m\times d}),\\
^{k}U^{t,x}:=(^{k}U^{i,t,x})_{i=1,m}\in\mathbb{H}^{2}(\mathbb{L}^{2}_{m}(\lambda_{k}));\\
(iii)~^{k}Y^{t,x}_{s}=g(^{k}X^{t,x}_{T})+\displaystyle\int^{T}_{s}f_{\mu_{k}}(r,^{k}X^{t,x}_{r},^{k}Y^{t,x}_{r},^{k}Z^{t,x}_{r},^{k}U^{t,x}_{r})\,dr\\
\quad\quad\quad\quad\quad\quad\quad\quad-\displaystyle\int^{T}_{s}\left\lbrace ^{k}Z^{t,x}_{r}\,dB_{r}+\displaystyle\int_{\mathrm{E}}\{^{k}U^{t,x}_{r}(e)\}\tilde{\mu}_{k}(dr,de)\right\rbrace ,\quad s\leq T;\\
(iv)~^{k}Y^{t,x}_{T}=g(^{k}X^{t,x}_{T}). 
\end{array}
\right.
\end{equation}
Finally let us introduce the following functions $(f^{(i)})_{i=1,m}$ defined by:
$\forall (t,x,y,z,\zeta)\in[0,T]\times\mathbb{R}^{k}\times\mathbb{R}^{m}\times\mathbb{R}^{m\times d}\times\mathbb{L}^{2}_{m}(\lambda_{k}),\\f_{\mu_{k}}(t,x,y,z,\zeta)=(f^{(i)}_{\mu_{k}}(t,x,y,z_{i},\zeta_{i}))_{i=1,m}:=\left( h^{(i)}\left(t,x,y,z,\displaystyle\int_{E}\gamma^{i}(t,x,e)\zeta_{i}(e)\lambda_{k}(de)\right)\right)_{i=1,m}$.\\
First let us emphasize that this latter BSDE is related to the filtration $(\mathcal{F}^{k}_{s})_{s\leq T}$ generated by the Brownian motion and the independent random measure $\mu_{k}$. However this point does not raise major issues since for any $s\leq T$, $\mathcal{F}^{k}_{s}\subset \mathcal{F}_{s}$ and thanks to the relationship between $\mu$ and $\mu_{k}$.\\
Next by the properties of the functions $b$, $\sigma$, $\beta$ and by the same opinions of proposition \ref{prop3.2} and proposition \ref{prop3.3}, there exists an unique quadriple $(^{k}Y^{t,x},^{k}Z^{t,x},^{k}U^{t,x})$ solving (\ref{4.46}) and there also exists a function $u^{k}$ from $[0,T]\times \mathbb{R}^{k}$ into $\mathbb{R}^{m}$ of $\Pi^{c}_{g}$ such that
\begin{equation}\label{4.47}
\forall s\in[t,T],~~^{k}Y^{t,x}:=u^{k}(s,^{k}X^{t,x}),~\mathbb{P}-a.s.
\end{equation}  
Moreover as in proposition \ref{prop4.2}, there exists positive constants $C$ and $\kappa$ wich do not depend on $k$ such that:
\begin{equation}\label{4.48}
\forall t,x,x',~~|u^{k}(t,x)-u^{k}(t,x')|\leq \rho\left(M_{2}\left|x-x'\right|^{2}(1+\left|x-x'\right|^{2})\right)\left[C(1+\left|x\right|^{\kappa})\right].
\end{equation}
Finally as $\lambda_{k}$ is finite then we have the following relationship between the process $^{k}U^{t,x}:=(^{k}U^{i;t,x})_{i=1,m}$ and the deterministics functions $u^{k}:=(u^{k}_{i})_{i=1,m}$ (see \cite{hamaMor}): $\forall i=1,\ldots,m$; 
$$^{k}U^{i;t,x}_{s}(e)=u^{k}_{i}(s,^{k}X^{t,x}_{s-}+\beta(s,^{k}X^{t,x}_{s-},e))-u^{k}_{i}(s,^{k}X^{t,x}_{s-}),~~d\mathbb{P}\otimes ds\otimes d\lambda_{k}-a.e.~\text{on }\Omega\times[t,T]\times E.$$
This is mainly due to the fact that $^{k}U^{t,x}$ belongs to $\mathbb{L}^{1}\cap\mathbb{L}^{2}(ds\otimes d\mathbb{P}\otimes d\lambda_{k})$ since $\lambda_{k}(E)<\infty$ and then we can split the stochastic integral w.r.t. $\tilde{\mu}_{k}$ in (\ref{4.46}). Therefore for all $i=1,\ldots,m$,
\begin{equation}\label{4.49}
^{k}U^{i;t,x}_{s}(e)1_{\{|e|\geq \frac{1}{k}\}}=(u^{k}_{i}(s,^{k}X^{t,x}_{s-}+\beta(s,^{k}X^{t,x}_{s-},e))-u^{k}_{i}(s,^{k}X^{t,x}_{s-}))1_{\{|e|\geq \frac{1}{k}\}},~~d\mathbb{P}\otimes ds\otimes d\lambda_{k}-a.e.~\text{on }\Omega\times[t,T]\times E.
\end{equation}
\end{proof}
\textbf{Step 2: Convergence of the auxiliary processes}\\
Let's now prove the following convergence result;
\begin{equation}\label{4.51}
\mathbb{E}\left[\sup_{s\leq T}\left|Y^{t,x}_{s}-^{k}Y^{t,x}_{s}\right|^{2}+\displaystyle\int^{T}_{0}\left|Z^{t,x}_{s}-^{k}Z^{t,x}_{s}\right|^{2}\,ds+\displaystyle\int^{T}_{0}\,ds\displaystyle\int_{E}\lambda(de)\left|U^{t,x}_{s}(e)-^{k}U^{t,x}_{s}(e)1_{\{|e|\geq \frac{1}{k}\}}\right|^{2}\right]\substack{\longrightarrow\\ k\longrightarrow+\infty}0;
\end{equation}
where $(Y^{t,x},Z^{t,x},U^{t,x})$ is solution of the BSDE with jumps (\ref{3.17}).\\
It should be noted that this convergence (\ref{4.51}) requires as the technique used in (\ref{prop4.2}) the following convergence:  
\begin{equation}\label{4.50}
\mathbb{E}\left[\sup_{s\leq T}\left|X^{t,x}_{s}-^{k}X^{t,x}_{s}\right|^{2}\right]\substack{\longrightarrow\\ k\longrightarrow+\infty}0.
\end{equation}
\begin{proof}[of \ref{4.50}]
\begin{eqnarray}
& {}{} & X^{t,x}_{s}-^{k}X^{t,x}_{s}=\displaystyle\int^{s}_{0}(b(r,X^{t,x}_{r})-b(r,^{k}X^{t,x}_{r}))\,dr+\displaystyle\int^{s}_{0}(\sigma(r,X^{t,x}_{r})-\sigma(r,^{k}X^{t,x}_{r}))\,dB_{r}\nonumber\\
& {}{} &\qquad\qquad+\displaystyle\int^{s}_{0}\displaystyle\int_{\mathrm{E}}(\beta(r,X^{t,x}_{r-},e)-\beta(r,^{k}X^{t,x}_{r-},e)1_{\{|e|\geq \frac{1}{k}\}})\tilde{\mu}_{k}\,(dr,de).\nonumber\\
\end{eqnarray}
Next let $\eta\in[0,T]$. Since $|a+b+c|^{2}\leq 3(|a|^{2}+|b|^{2}+|c|^{2})$ for any real constants $a$, $b$ and $c$ and by the Cauchy-Schwartz and Burkholder-Davis-Gundy inequalities we have:
\newpage
\begin{eqnarray}\label{4.36n}
& {}{} & \mathbb{E}\left[\sup_{0\leq s\leq \eta}\left|X^{t,x}_{s}-^{k}X^{t,x}_{s}\right|^{2}\right]\nonumber\\
{} & \leq & 3\mathbb{E}\left[\sup_{0\leq s\leq \eta}\left|\displaystyle\int^{s}_{0}(b(r,X^{t,x}_{r})-b(r,^{k}X^{t,x}_{r}))\,dr\right|^{2}+\sup_{0\leq s\leq \eta}\left|\displaystyle\int^{s}_{0}(\sigma(r,X^{t,x}_{r})-\sigma(r,^{k}X^{t,x}_{r}))\,dB_{r}\right|^{2}
\right.\nonumber\\
& {}{} &\left.
\qquad\qquad+\sup_{0\leq r\leq \eta}\left|\displaystyle\int^{r}_{0}\displaystyle\int_{\mathrm{E}}(\beta(r,X^{t,x}_{r-},e)-\beta(r,^{k}X^{t,x}_{r-},e)1_{\{|e|\geq \frac{1}{k}\}})\tilde{\mu}_{k}\,(dr,de)\right|^{2}\right]\nonumber\\
{} & \leq & C\mathbb{E}\left[\displaystyle\int^{\eta}_{0}\sup_{0\leq \eta\leq r}\left\lbrace \left|(b(r,X^{t,x}_{r})-b(r,^{k}X^{t,x}_{r}))\right|^{2}+\left|(\sigma(r,X^{t,x}_{r})-\sigma(r,^{k}X^{t,x}_{r}))\right|^{2}\right\rbrace\,dr\right]\nonumber\\
&{}{}&+C\mathbb{E}\left[\displaystyle\int^{\eta}_{0}\displaystyle\int_{\mathrm{E}}\sup_{0\leq r\leq \eta}\left|\beta(r,X^{t,x}_{r-},e)-\beta(r,^{k}X^{t,x}_{r-},e)\right|^{2}\,\lambda_{k}(de)dr
\right.\nonumber\\
 & {}{} &\left.
 +\displaystyle\int^{\eta}_{0}\displaystyle\int_{\mathrm{E}}\sup_{0\leq r\leq \eta}\left|\beta(r,X^{t,x}_{r-},e)\right|^{2}1_{\{|e|<\frac{1}{k}\}}\,\lambda(de)dr\right]\nonumber\\
\end{eqnarray}
Since $b$, $\sigma$ and  $\beta$ belong to $\mathcal{M}$, then we have: $\forall r\in[0,T]$,
\begin{equation} 
\sup_{0\leq \tau\leq r}\left\lbrace \left|(b(r,X^{t,x}_{\tau})-b(r,^{k}X^{t,x}_{\tau}))\right|^{2}+\left|(\sigma(r,X^{t,x}_{\tau})-\sigma(r,^{k}X^{t,x}_{\tau}))\right|^{2}\right\rbrace \leq C\rho(\left|X^{t,x}_{\tau}-^{k}X^{t,x}_{\tau}\right|^{2})
\end{equation}
and
\begin{equation}
\displaystyle\int_{\mathrm{E}}\sup_{0\leq r\leq \eta}\left|\beta(r,X^{t,x}_{r-},e)-\beta(r,^{k}X^{t,x}_{r-},e)\right|^{2}\,\lambda_{k}(de)\leq  C\rho(\left|X^{t,x}_{\tau}-^{k}X^{t,x}_{\tau}\right|^{2})
\end{equation}
Plug now those two last inequalities in the previous one to obtain: $\forall \eta\in[0,T]$,
\begin{eqnarray*}
\mathbb{E}\left[\sup_{0\leq s\leq \eta}\left|X^{t,x}_{s}-^{k}X^{t,x}_{s}\right|^{2}\right] &\leq & C\mathbb{E}\left[\displaystyle\int^{\eta}_{0}\rho(\left|X^{t,x}_{\tau}-^{k}X^{t,x}_{\tau}\right|^{2})\,dr+\displaystyle\int_{\{|e|<\frac{1}{k}\}}(1\wedge|e|^{2})\,\lambda(de)\right]\\
&\leq & C\displaystyle\int^{\eta}_{0}\rho(\mathbb{E}\left[\left|X^{t,x}_{\tau}-^{k}X^{t,x}_{\tau}\right|^{2}\right])\,dr+C\mathbb{E}\left[\displaystyle\int_{\{|e|<\frac{1}{k}\}}(1\wedge|e|^{2})\,\lambda(de)\right]~~(\text{by Jensen}), 
\end{eqnarray*}
By Bihari's inequality (see \cite{fan} page 171 and \cite{pa}) and the fact of $\displaystyle\int_{\{|e|<\frac{1}{k}\}}(1\wedge|e|^{2})\,\lambda(de)\substack{\longrightarrow\\ k\longrightarrow+\infty}0$; we obtain our result the (\ref{4.50}).
\end{proof}
\\
We now focus on (\ref{4.51}). Note that we can apply Ito's formula, even if the BSDEs are related to filtrations and Poisson random measures which are not the same, since:\\
(i) $\mathcal{F}^{k}_{s}\subset\mathcal{F}_{s}$, $\forall s\leq T$;\\
(ii) for any $s\leq T$, $\displaystyle\int^{s}_{0}\displaystyle\int_{\mathrm{E}}^{k}U^{i;t,x}(e)\tilde{\mu}_{k}\,(dr,de)=\displaystyle\int^{s}_{0}\displaystyle\int_{\mathrm{E}}^{k}U^{i;t,x}(e)1_{\{|e|\geq\frac{1}{k}\}}\tilde{\mu}\,(dr,de)$ and then the first $(\mathcal{F}^{k}_{s})_{s\leq T}-$martingale is also an $(\mathcal{F}_{s})_{s\leq T}-$martingale.
$\forall s\in[0,T]$,
\begin{eqnarray}
& {}{} &\left|\vec{Y}^{t,x}_{s}-^{k}Y^{t,x}_{s}\right|^{2}+\displaystyle\int^{T}_{0}\left|Z^{t,x}_{s}-^{k}Z^{t,x}_{s}\right|^{2}\,ds+\sum_{s\leq r\leq T}(^{k}\Delta_{r}\vec{Y}^{t,x}_{r})^{2}\nonumber\\
& {}{} &=\left|g(X^{t,x}_{T})-g(^{k}X^{t,x}_{T})\right|^{2}+2\displaystyle\int^{T}_{s}
\left(\vec{Y}^{t,x}_{r}-^{k}Y^{t,x}_{r}\right)\times ^{k}\Delta f(r)\,dr\nonumber \\
& {}{} &-2\displaystyle\int^{T}_{s}\displaystyle\int_{\mathrm{E}}\left(\vec{Y}^{t,x}_{r}
-^{k}Y^{t,x}_{r}\right)\left(^{k}\Delta U_{r}(e)\right)\tilde{\mu}(\mathrm{d}r,\mathrm{d}e)-2\displaystyle\int^{T}_{s} \left(\vec{Y}^{t,x}_{r}-^{k}\vec{Y}^{t,x}_{r}\right)\left(^{k}\Delta Z_{r}\right)\,dB_{r};\nonumber
\end{eqnarray}
and taking expectation we obtain: $\forall s\in[t,T]$,

\begin{eqnarray}\label{4.53}
& {}{} &\mathbb{E}\left[\left|\vec{Y}^{t,x}_{s}-^{k}Y^{t,x}_{s}\right|^{2}+\displaystyle\int^{T}_{0}\left\lbrace\left|Z^{t,x}_{s}-^{k}Z^{t,x}_{s}\right|^{2}+\displaystyle\int_{E}\left|U^{t,x}_{s}-^{k}U^{t,x}_{s}1_{\{|e|\geq\frac{1}{k}\}}\right|^{2}\,\lambda(de)\right\rbrace\,ds\right]\nonumber\\
& {}{} &\leq\mathbb{E}\left[\left|g(X^{t,x}_{T})-g(^{k}X^{t,x}_{T})\right|^{2}+2\displaystyle\int^{T}_{s}
\left(\vec{Y}^{t,x}_{r}-^{k}Y^{t,x}_{r}\right)\times ^{k}\Delta f(r)\,dr\right];\nonumber\\
\end{eqnarray}
where the processes $^{k}\Delta X_{r}$, $^{k}\Delta Y_{r}$, $^{k}\Delta f(r)$, $^{k}\Delta Z_{r}$ and $^{k}\Delta U_{r}$ are defined as follows: $\forall r\in[0,T]$,\\
$^{k}\Delta f(r):=((^{k}\Delta f^{(i)}(r))_{i=1,m}=(f^{(i)}(r,X^{t,x}_{r},\vec{Y}^{t,x}_{r},Z^{i;t,x}_{r},U^{i;t,x}_{r})-f^{(i)}_{k}(r,^{k}X^{t,x}_{r},^{k}Y^{t,x}_{r},^{k}Z^{t,x}_{r},^{k}U^{t,x}_{r}))_{i=1,m}$,
$^{k}\Delta X_{r}=X^{t,x}_{r}-^{k}X^{t,x}_{r}$, $^{k}\Delta Y(r)=\vec{Y}^{t,x}_{r}-^{k}Y^{t,x}_{r}=(Y^{j;t,x}_{r}-^{k}Y^{j;t,x}_{r})_{j=1,m}$, $^{k}\Delta Z_{r}=Z^{t,x}_{r}-^{k}Z^{t,x}_{r}$ and $^{k}\Delta U_{r}=U^{t,x}_{r}-^{k}U^{t,x}_{s}1_{\{|e|\geq\frac{1}{k}\}}$.\\

Next let us set for $r\leq T$,
$$^{k}\Delta f(r)=(f(r,X^{t,x}_{r},\vec{Y}^{t,x}_{r},Z^{t,x}_{r},U^{t,x}_{r})-f_{k}(r,^{k}X^{t,x}_{r},^{k}Y^{t,x}_{r},^{k}Z^{t,x}_{r},^{k}U^{t,x}_{r}))=A(r)+B(r)+C(r)+D(r);$$
where  for any $i=1,\ldots,m$, 
\begin{eqnarray*}
A(r) & = & \left(h^{(i)}\left(r,X^{t,x}_{r},\vec{Y}^{t,x}_{r},Z^{i;t,x}_{r},\displaystyle\int_{E}\gamma^{i}(r,X^{t,x}_{r},e)U^{i;t,x}_{r}(e)\lambda(de)\right)
\right.\\
&{}{}&\left.
-h^{(i)}\left(r,^{k}X^{t,x}_{r},\vec{Y}^{t,x}_{r},Z^{i;t,x}_{r},\displaystyle\int_{E}\gamma^{i}(r,X^{t,x}_{r},e)U^{i;t,x}_{r}(e)\lambda(de)\right)\right)_{i=1,m};\\ 
B(r) & = & \left(h^{(i)}\left(r,^{k}X^{t,x}_{r},\vec{Y}^{t,x}_{r},Z^{i;t,x}_{r},\displaystyle\int_{E}\gamma^{i}(r,X^{t,x}_{r},e)U^{i;t,x}_{r}(e)\lambda(de)\right)
\right.\\
&{}{}&\left.
-h^{(i)}\left(r,^{k}X^{t,x}_{r},^{k}Y^{t,x}_{r},Z^{i;t,x}_{r},\displaystyle\int_{E}\gamma^{i}(r,X^{t,x}_{r},e)U^{i;t,x}_{r}(e)\lambda(de)\right)\right)_{i=1,m};\\ 
C(r) & = & \left(h^{(i)}\left(r,^{k}X^{t,x}_{r},^{k}Y^{t,x}_{r},Z^{i;t,x}_{r},\displaystyle\int_{E}\gamma^{i}(r,X^{t,x}_{r},e)U^{i;t,x}_{r}(e)\lambda(de)\right)
\right.\\
&{}{}&\left.
-h^{(i)}\left(r,^{k}X^{t,x}_{r},^{k}Y^{t,x}_{r},^{k}Z^{i;t,x}_{r},\displaystyle\int_{E}\gamma^{i}(r,X^{t,x}_{r},e)U^{i;t,x}_{r}(e)\lambda(de)\right)\right)_{i=1,m};\\ 
D(r) & = & \left( h^{(i)}\left(r,^{k}X^{t,x}_{r},^{k}Y^{t,x}_{r},^{k}Z^{i;t,x}_{r},\displaystyle\int_{E}\gamma^{i}(r,X^{t,x}_{r},e)U^{i;t,x}_{r}(e)\lambda(de)\right)
\right.\\
&{}{}&\left.
-h^{(i)}\left(r,^{k}X^{t,x}_{r},^{k}Y^{t,x}_{r},^{k}Z^{i;t,x}_{r},\displaystyle\int_{E}\gamma^{i}(r,^{k}X^{t,x}_{r},e)^{k}U^{i;t,x}_{r}(e)\lambda_{k}(de)\right)\right)_{i=1,m}.
\end{eqnarray*}
Since $g$ belongs to $\mathcal{M}$ and by (\ref{4.50}) we have,
\begin{equation}\label{4.54}
\mathbb{E}\left[\left|g(X^{t,x}_{T})-g(^{k}X^{t,x}_{T})\right|^{2}\right]\substack{\displaystyle\longrightarrow 0\\ k\rightarrow+\infty}
\end{equation}
Now we will interest to $\mathbb{E}\left[\displaystyle\int^{T}_{s}
\left(\vec{Y}^{t,x}_{r}-^{k}Y^{t,x}_{r}\right)\times ^{k}\Delta f(r)\,dr\right]$ for found (\ref{4.51}).\\
By (\ref{2.11}), we have: $\forall r\in[0,T]$
\begin{eqnarray}\label{4.56} 
\left|A(r)\right|^{2} &\leq & \rho(\left|X^{t,x}_{r}-^{k}X^{t,x}_{r}\right|^{2});\nonumber\\
\left|B(r)\right|+\left|C(r)\right| &\leq & C\{\left|\vec{Y}^{t,x}_{r}-^{k}Y^{t,x}_{r}\right|+\left|Z^{t,x}_{r}-^{k}Z^{t,x}_{r}\right|\}.\nonumber\\
\end{eqnarray}
Now let us deal with $D(r)$ which is more involved. First note that $D(r)=(D_{i}(r))_{i=1,m}$ where 
\begin{eqnarray*}
D_{i}(r)& = & h^{(i)}\left(r,^{k}X^{t,x}_{r},^{k}Y^{t,x}_{r},^{k}Z^{i;t,x}_{r},\displaystyle\int_{E}\gamma^{i}(r,X^{t,x}_{r},e)U^{i;t,x}_{r}(e)\lambda(de)\right)\\
&{}{}&-h^{(i)}\left(r,^{k}X^{t,x}_{r},^{k}Y^{t,x}_{r},^{k}Z^{i;t,x}_{r},\displaystyle\int_{E}\gamma^{i}(r,^{k}X^{t,x}_{r},e)^{k}U^{i;t,x}_{r}(e)\lambda_{k}(de)\right).
\end{eqnarray*}
But as $h^{(i)}$ is Lipschitz w.r.t to the last component $q$ and by the relation (\ref{2.12}) then ,
\begin{eqnarray}\label{4.57}
\left|D(r)\right|^{2} & \leq & C\left\lbrace \displaystyle\int_{E}\left|\gamma^{i}(r,X^{t,x}_{r},e)U^{i;t,x}_{r}(e)-\gamma^{i}(r,^{k}X^{t,x}_{r},e)^{k}U^{i;t,x}_{r}(e)1_{\{|e|\geq\frac{1}{k}\}}\right|\lambda(de)\right\rbrace^{2}\nonumber\\
{}&\leq & C\left\lbrace \left\lbrace \displaystyle\int_{E}\left|\gamma^{i}(r,X^{t,x}_{r},e)-\gamma^{i}(r,^{k}X^{t,x}_{r},e)\right|\left|U^{i;t,x}_{r}(e)\right|\,\lambda(de)\right\rbrace^{2}
\right.\nonumber\\
&{}{}&\left.
+\left\lbrace \displaystyle\int_{E}\left|\gamma^{i}(r,X^{t,x}_{r},e)\right|\left| U^{i;t,x}_{r}(e)-^{k}U^{i;t,x}_{r}(e)1_{\{|e|\geq\frac{1}{k}\}}\right|\lambda(de)\right\rbrace^{2}\right\rbrace\nonumber\\
{}&\leq & C\rho(\left|X^{t,x}_{r}-^{k}X^{t,x}_{r}\right|^{2})\left\lbrace\displaystyle\int_{E}\left|U^{i;t,x}_{r}(e)\right|^{2}\lambda(de)\right\rbrace\nonumber\\
&{}{}&+C\displaystyle\int_{E}(1\wedge|e|)^{2}\lambda(de)\times\displaystyle\int_{E}\left| U^{i;t,x}_{r}(e)-^{k}U^{i;t,x}_{r}(e)1_{\{|e|\geq\frac{1}{k}\}}\right|^{2}\lambda(de),
\end{eqnarray}
By using the majorations obtain in (\ref{4.56}) and in (\ref{4.57}) and Cauchy-Schwartz inequality;
\begin{eqnarray}\label{4.58}
& {}{} &\mathbb{E}\left[\left|\vec{Y}^{t,x}_{s}-^{k}Y^{t,x}_{s}\right|^{2}+\displaystyle\int^{T}_{0}\left\lbrace\left|Z^{t,x}_{s}-^{k}Z^{t,x}_{s}\right|^{2}+\displaystyle\int_{E}\left|U^{t,x}_{s}-^{k}U^{t,x}_{s}1_{\{|e|\geq\frac{1}{k}\}}\right|^{2}\,\lambda(de)\right\rbrace\,ds\right]\nonumber\\
& {}{} &\leq\mathbb{E}\left[\left|g(X^{t,x}_{T})-g(^{k}X^{t,x}_{T})\right|^{2}\right]+C_{\epsilon}\mathbb{E}\left[\displaystyle\int^{T}_{s}
\left|\vec{Y}^{t,x}_{r}-^{k}Y^{t,x}_{r}\right|^{2}\right]+\displaystyle\int^{T}_{s}\rho(\mathbb{E}\left[\left|X^{t,x}_{r}-^{k}X^{t,x}_{r}\right|^{2}\right])\,dr\nonumber\\
& {}{} &+C\sqrt{\mathbb{E}\left[\sup_{s\leq r\leq T}\rho^{2}(\left|X^{t,x}_{r}-^{k}X^{t,x}_{r}\right|^{2})\right]}\times\sqrt{\mathbb{E}\left[\left(\displaystyle\int^{T}_{s}\displaystyle\int_{E}\left|U^{i;t,x}_{r}(e)\right|^{2}\lambda(de)\,dr\right)^{2}\right]}\nonumber\\
& {}{} &+C\epsilon\mathbb{E}\left[\displaystyle\int^{T}_{t}\displaystyle\int_{E}\left|U^{t,x}_{s}-^{k}U^{t,x}_{s}1_{\{|e|\geq\frac{1}{k}\}}\right|^{2}\,\lambda(de)\,ds\right]
\end{eqnarray}
For $C\epsilon<1$ and the Gronwall's lemma going through the dominated convergence theorem, the continuity of $g$ and $\rho$ and the lemma \ref{lem4.1}, then
\begin{equation}\label{4.59}
\mathbb{E}\left[\left|\vec{Y}^{t,x}_{s}-^{k}Y^{t,x}_{s}\right|^{2}\right]\substack{\displaystyle\longrightarrow 0\\ k\rightarrow+\infty}
\end{equation}
and in taking $s=t$ we obtain $u^{k}(t,x)\substack{\displaystyle\longrightarrow u(t,x)\\ k\rightarrow+\infty}$. As $(t,x)\in[0,T]\times \mathbb{R}^{k}$ is arbitrary then u$^{k}\substack{\displaystyle\longrightarrow u\\ k\rightarrow+\infty}$ pointwisely.\\
Taking the same arguments as when getting (\ref{4.59}); we once again add Lebesgue's dominated convergence theorem to have,
\begin{equation}\label{4.60}
\mathbb{E}\left[\displaystyle\int^{T}_{t}\displaystyle\int_{E}\left|U^{t,x}_{s}-^{k}U^{t,x}_{s}1_{\{|e|\geq\frac{1}{k}\}}\right|^{2}\,\lambda(de)\,ds\right]\substack{\displaystyle\longrightarrow 0\\ k\rightarrow+\infty}.
\end{equation}
\textbf{Step 3: Conclusion}\\
First note that by (\ref{4.48}) and the pointwise convergence of $(u^{k})_{k}$ to $u$, if $(x_{k})_{k}$ is a sequence of $\mathbb{R}^{k}$ which converges to $x$ then $((u^{k}(t,x_{k}))_{k})$ converges to $u(t,x)$.\\
Now let us consider a subsequence which we still denote by $\{k\}$ such that $\sup_{s\leq T}\left|X^{t,x}_{s}-^{k}X^{t,x}_{s}\right|^{2}\substack{\displaystyle\longrightarrow 0\\ k\rightarrow+\infty}$, $\mathbb{P}$-a.s. (and then $\left|X^{t,x}_{s-}-^{k}X^{t,x}_{s-}\right|\substack{\displaystyle\longrightarrow 0\\ k\rightarrow+\infty}$ since $\left|X^{t,x}_{s-}-^{k}X^{t,x}_{s-}\right|\leq \sup_{s\leq T}\left|X^{t,x}_{s}-^{k}X^{t,x}_{s}\right|^{2}$). By (\ref{4.50}), this subsequence exists. As the mapping $x\mapsto \beta(t,x,e)$ is Lipschitz then the sequence
\begin{eqnarray}\label{4.61}
&{}{}&\left(^{k}U^{t,x}_{s}(e)1_{\{|e|\geq\frac{1}{k}\}}\right)_{k}=\left((u^{k}_{i}(s,^{k}X^{t,x}_{s-}+\beta(s,^{k}X^{t,x}_{s-},e))-u^{k}_{i}(s,^{k}X^{t,x}_{s-}))1_{\{|e|\geq\frac{1}{k}\}}\right)_{k\geq 1}\substack{\displaystyle\longrightarrow {}\\ k\rightarrow+\infty}\nonumber\\
&{}{}& (u_{i}(s,X^{t,x}_{s-}+\beta(s,X^{t,x}_{s-},e))-u_{i}(s,X^{t,x}_{s-})),\quad d\mathbb{P}\otimes ds\otimes d\lambda-a.e.\quad \text{on }\Omega\times[t,T]\times E\quad
\end{eqnarray}
for any $i=1,\ldots,m$. Finally from (\ref{4.60}) we deduce that
\begin{equation}\label{4.62}
U^{t,x}_{s}(e)=(u_{i}(s,X^{t,x}_{s-}+\beta(s,X^{t,x}_{s-},e))-u_{i}(s,X^{t,x}_{s-})),\quad\text{on }\Omega\times[t,T]\times E
\end{equation}
which is the desired result.

\section{The main result}
Unlike Barles and al.\cite{bar} our result on viscosity solutions is established for the following definition.
\begin{definition}\label{def5.1}
We say that a family of deterministics functions $u=(u^{i})_{i=1,m}$ which belongs to $\mathcal{M}\quad\forall i\in\{1,\ldots,m\}$ is a viscosity sub-solution (resp. super-solution) of the IPDE (\ref{eq1}) if:\\
$(i)\quad \forall x\in\mathbb{R}^{k}$, $u^{i}(x,T)\leq g^{i}(x)$ (resp. $u^{i}(x,T)\geq g^{i}(x)$);\\
$(ii)\quad\text{For any } (t,x)\in[0,T]\times\mathbb{R}^{k}$ and any function $\phi$ of class $C^{1,2}([0,T]\times\mathbb{R}^{k})$ such that $(t,x)$ is a global maximum  point of $u^{i}-\phi$ (resp. global minimum  point of $u^{i}-\phi$) and $(u^{i}-\phi)(t,x)=0$ one has
\begin{equation}\label{5.63}
-\partial_{t}\phi(t,x)-\mathcal{L}^{X}\phi(t,x)-h^{i}(t,x,(u^{j}(t,x))_{j=1,m},\sigma^{\top}(t,x))D_{x}\phi(t,x),B_{i}u^{i}(t,x))\leq 0 
\end{equation}
$\left(resp.
\right.$
\begin{equation}\label{5.64}
\left.
-\partial_{t}\phi(t,x)-\mathcal{L}^{X}\phi(t,x)-h^{i}(t,x,(u^{j}(t,x))_{j=1,m},\sigma^{\top}(t,x))D_{x}\phi(t,x),B_{i}u^{i}(t,x))\geq 0\right).
\end{equation}
The family $u=(u^{i})_{i=1,m}$ is a viscosity solution of (\ref{eq1}) if it is both a viscosity sub-solution and viscosity super-solution.\\
Note that $\mathcal{L}^{X}\phi(t,x)=b(t,x)^{\top}\mathrm{D}_{x}\phi(t,x)+\frac{1}{2}\mathrm{Tr}(\sigma\sigma^{\top}(t,x)\mathrm{D}^{2}_{xx}\phi(t,x))+\mathrm{K}\phi(t,x)$;\\
where $\mathrm{K}\phi(t,x)=\displaystyle\int_{\mathrm{E}}(\phi(t,x+\beta(t,x,e))-\phi(t,x)-\beta(t,x,e)^{\top}\mathrm{D}_{x}\phi(t,x))\lambda(de).$
\end{definition}
\begin{theorem}
Under assumptions (\textbf{H1}) and (\textbf{H2}), the IPDE (\ref{eq1}) has unique solution which is the $m$-tuple of functions $(u^{i})_{i=1,m}$ defined in proposition \ref{prop3.3} by (\ref{3.18}).  
\end{theorem}
\newpage
\begin{proof}
\emph{\underline{Step $1$}:} \emph{Existence}\\
Assume that assumptions (\textbf{H1}) and (\textbf{H2}) are fulfilled, then the following multi-dimensional BSDEs with jumps
\begin{equation}\label{5.65}
\left
\{\begin{array}{ll}
(i)~\underline{\vec{Y}}^{t,x}:=(\underline{Y}^{i;t,x})_{i=1,m}\in\mathcal{S}^{2}(\mathbb{R}^{m}),~\underline{Z}^{t,x}:=(\underline{Z}^{i;t,x})_{i=1,m}\in\mathbb{H}^{2}(\mathbb{R}^{m\times d}),~\underline{U}^{t,x}:=(\underline{U}^{i;t,x})_{i=1,m}\in\mathbb{H}^{2}(\mathbb{L}^{2}_{m}(\lambda));\\
(ii)~\underline{Y}^{i;t,x}_{s}= g^{i}(X^{t,x}_{T})-\displaystyle\int^{T}_{s}\underline{Z}^{i;t,x}\mathrm{d}\mathrm{B}_{r}-\displaystyle\int^{T}_{s}
\displaystyle\int_{\mathrm{E}}\underline{U}^{i;t,x}_{r}(e)\tilde{\mu}(\mathrm{d}r,\mathrm{d}e).\\
\quad\quad\quad+\displaystyle\int^{T}_{s}h^{(i)}(r,X^{t,x}_{r},\underline{Y}^{i;t,x}_{r},\underline{Z}^{i;t,x}_{r},\displaystyle\int_{\mathrm{E}}\gamma^{i}(t,X^{t,x}_{r},e)\{(u^{i}(t,X^{t,x}_{r-}+\beta(t,X^{t,x}_{r-},e))-u^{i}(t,X^{t,x}_{r-}))\}\,\lambda(de))dr\\
(iii)~\underline{Y}^{i;t,x}_{T}= g^{i}(X^{t,x}_{T}).
\end{array}
\right.
\end{equation}
has unique solution  $(\underline{Y},\underline{Z},\underline{U})$.
Next as for any $i=1,\ldots,m$, $u^{i}$ belongs to $\mathcal{M}$, then by proposition \ref{prop3.3} the (\ref{3.18}), there exists a family of deterministics continuous functions of polynomial growth $(\underline{u}^{i})_{i=1,m}$ that fact for any $(t,x)\in[0,T]\times\mathbb{R}^{k}$,
$$\forall s\in[t,T],\qquad \underline{Y}^{i;t,x}_{s}=\underline{u}^{i}(s,X^{t,x}_{s}).$$
Such that by the same proposition, the family $(\underline{u}^{i})_{i=1,m}$ is a viscosity solution of the following system:
\begin{equation}\label{5.66}
\left
\{\begin{array}{ll}
-\partial_{t}\underline{u}^{i}(t,x)-b(t,x)^{\top}\mathrm{D}_{x}\underline{u}^{i}(t,x)-\frac{1}{2}\mathrm{Tr}(\sigma\sigma^{\top}(t,x)\mathrm{D}^{2}_{xx}\underline{u}^{i}(t,x))\\
\quad\quad-\mathrm{K}_{i}\underline{u}^{i}(t,x)-\mathit{h}^{(i)}(t,x,(\underline{u}^{j}(t,x))_{j=1,m},(\sigma^{\top}\mathrm{D}_{x}\underline{u}^{i})(t,x),\mathrm{B}_{i}u^{i}(t,x))=0,\quad (t,x)\in\left[ 0,T\right] \times\mathbb{R}^{k};\\
u^{i}(T,x)=g^{i}(x).
\end{array}
\right.
\end{equation}
Now we have the family $(\underline{u}^{i})_{i=1,m}$ is a viscosity solution, our main objective is to found relation between $(\underline{u}^{i})_{i=1,m}$ and $(u^{i})_{i=1,m}$ which is defined in (\ref{3.18}).\\
For this, let us consider the system of BSDE with jumps
\begin{equation}\label{5.67}
\left
\{\begin{array}{ll}
(i)~\vec{Y}^{t,x}:=(Y^{i;t,x})_{i=1,m}\in\mathcal{S}^{2}(\mathbb{R}^{m}),~Z^{t,x}:=(Z^{i;t,x})_{i=1,m}\in\mathbb{H}^{2}(\mathbb{R}^{m\times d}),~\\U^{t,x}:=(U^{i;t,x})_{i=1,m}\in\mathbb{H}^{2}(\mathbb{L}^{2}_{m}(\lambda));\\
(ii)~Y^{i;t,x}_{s}= g^{i}(X^{t,x}_{T})-\displaystyle\int^{T}_{s}Z^{i;t,x}\mathrm{d}
\mathrm{B}_{r}-\displaystyle\int^{T}_{s}
\displaystyle\int_{\mathrm{E}}U^{i;t,x}_{r}(e)\tilde{\mu}(\mathrm{d}r,\mathrm{d}e).\\
\quad\quad\quad+\displaystyle\int^{T}_{s}h^{(i)}(r,X^{t,x}_{r},Y^{i;t,x}_{r},\underline{Z}^{i;t,x}_{r},\displaystyle\int_{\mathrm{E}}\gamma^{i}(t,X^{t,x}_{r},e)U^{i;t,x}_{r}(e)\,\lambda(de))dr;\\
(iii)~Y^{i;t,x}_{T}=g(X^{t,x}_{T}). 
\end{array}
\right.
\end{equation}
By uniqueness of the solution of the BSDEs with jumps (\ref{5.64}), that for any $s\in[t,T]$ and $\forall i\in\{1\ldots,m\}$, $\underline{Y}^{i;t,x}_{s}=Y^{i;t,x}_{s}$.\\
Therefore  $\underline{u}^{i}=u^{i}$, such that by (\ref{4.60}) we obtain $U^{t,x}_{s}(e)=(u_{i}(s,X^{t,x}_{s-}+\beta(s,X^{t,x}_{s-},e))-u_{i}(s,X^{t,x}_{s-})),\quad\text{on }\Omega\times[t,T]\times E$, which give the viscosity solution in the sense of definition $5.1$ (see \cite{hama}) by pluging (\ref{4.61}) in $h^{(i)}$ of (\ref{5.66}).

\emph{\underline{Step $2$}:} \emph{Uniqueness}\\

For uniqueness, let $(\overline{u}^{i})_{i=1,m}$ be another family of $\mathcal{M}$ which is solution viscosity of the system (\ref{eq1}) in the sense of definition $5.1$ and we consider BSDE with jumps defined with $\overline{u}^{i}$.   
\begin{equation}\label{5.68}
\left
\{\begin{array}{ll}
(i)~\vec{\overline{Y}}^{t,x}:=(\overline{Y}^{i;t,x})_{i=1,m}\in\mathcal{S}^{2}(\mathbb{R}^{m}),~\overline{Z}^{t,x}:=(\overline{Z}^{i;t,x})_{i=1,m}\in\mathbb{H}^{2}(\mathbb{R}^{m\times d}),~\overline{U}^{t,x}:=(\overline{U}^{i;t,x})_{i=1,m}\in\mathbb{H}^{2}(\mathbb{L}^{2}_{m}(\lambda));\\
(ii)~\overline{Y}^{i;t,x}_{s}= g^{i}(X^{t,x}_{T})-\displaystyle\int^{T}_{s}\overline{Z}^{i;t,x}\mathrm{d}\mathrm{B}_{r}-\displaystyle\int^{T}_{s}\displaystyle\int_{\mathrm{E}}\overline{U}^{i;t,x}_{r}(e)\tilde{\mu}(\mathrm{d}r,\mathrm{d}e).\\
\quad\quad\quad+\displaystyle\int^{T}_{s}h^{(i)}(r,X^{t,x}_{r},\overline{Y}^{i;t,x}_{r},\overline{Z}^{i;t,x}_{r},\displaystyle\int_{\mathrm{E}}\gamma^{i}(t,X^{t,x}_{r},e)(\overline{u}_{i}(s,X^{t,x}_{s-}+\beta(s,X^{t,x}_{s-},e))-\overline{u}_{i}(s,X^{t,x}_{s-}))\,\lambda(de))dr;\\
(iii)~\overline{Y}^{i;t,x}_{T}=g(X^{t,x}_{T}). 
\end{array}
\right.
\end{equation}
By Feynman Kac formula $\overline{u}^{i}(s,X^{t,x}_{s})=Y^{i;t,x}_{s}$ where $Y^{i;t,x}_{s}$ satisfies the BSDE with jumps (\ref{eq2}) associated to IPDE (\ref{eq1}).\\
Since that the BSDE with jumps (\ref{5.66}) has solution and it is unique by   
assumed that (\textbf{H1}) and (\textbf{H2}) are verified. By proposition \ref{prop3.3} the (\ref{3.18}), there exists a family of deterministic continuous functions of polynomial growth $(v^{i})_{i=1,m}$ that fact for any $(t,x)\in[0,T]\times\mathbb{R}^{k}$,
$$\forall s\in[t,T],\qquad \overline{Y}^{i;t,x}_{s}=v^{i}(s,X^{t,x}_{s}).$$
Such that by the same proposition, the family $(v^{i})_{i=1,m}$ is a viscosity solution of the following system: 
\begin{equation}\label{5.69}
\left
\{\begin{array}{ll}
-\partial_{t}v^{i}(t,x)-b(t,x)^{\top}\mathrm{D}_{x}v^{i}(t,x)-\frac{1}{2}\mathrm{Tr}(\sigma\sigma^{\top}(t,x)\mathrm{D}^{2}_{xx}v^{i}(t,x))\\
\quad\quad-\mathrm{K}_{i}v^{i}(t,x)-\mathit{h}^{(i)}(t,x,(v^{j}(t,x))_{j=1,m},(\sigma^{\top}\mathrm{D}_{x}v^{i})(t,x),\mathrm{B}_{i}\overline{u}^{i}(t,x))=0,\quad (t,x)\in\left[ 0,T\right] \times\mathbb{R}^{k};\\
u^{i}(T,x)=g^{i}(x).
\end{array}
\right.
\end{equation}
By uniqueness of solution of (\ref{5.67}) $\overline{u}^{i}$ is viscosity solution of (\ref{5.68}); and by proposition \ref{prop3.3} $v^{i}=\overline{u}^{i}$ $\forall i\in\{1,\ldots,m\}$.\\
Now for completing our proof we show that on $\Omega\times[t,T]\times E$, $ds\otimes d\mathbb{P}\otimes d\lambda-\text{a.e.}\quad\forall i\in\{1,\ldots, m\}$;
\begin{eqnarray}\label{5.71}
\overline{U}^{i;t,x}_{s}(e) & = & (v^{i}(s,X^{t,x}_{s-}+\beta(s,X^{t,x}_{s-},e))-v^{i}(s,X^{t,x}_{s-}))\nonumber\\
{} & = & (\overline{u}_{i}(s,X^{t,x}_{s-}+\beta(s,X^{t,x}_{s-},e))-\overline{u}_{i}(s,X^{t,x}_{s-})).
\end{eqnarray} 
By Remark $3.4$ in \cite{hama}; let us considere $(x_{k})_{k\geq 1}$ a sequence of $\mathbb{R}^{k}$ which converges to $x\in\mathbb{R}^{k}$ and the two following BSDE with jumps (adaptation is w.r.t. $\mathcal{F}^{k}$):
\begin{equation}\label{5.72}
\left
\{\begin{array}{ll}
(i)~\vec{\overline{Y}}^{k,t,x}:=(\overline{Y}^{i;k,t,x})_{i=1,m}\in\mathcal{S}^{2}(\mathbb{R}^{m}),~\overline{Z}^{k,t,x}:=(\overline{Z}^{i;k,t,x})_{i=1,m}\in\mathbb{H}^{2}(\mathbb{R}^{m\times d}),~\overline{U}^{k,t,x}:=(\overline{U}^{i;k,t,x})_{i=1,m}\in\mathbb{H}^{2}(\mathbb{L}^{2}_{m}(\lambda));\\
(ii)~\overline{Y}^{i;k,t,x}_{s}= g^{i}(X^{k,t,x}_{T})-\displaystyle\int^{T}_{s}\overline{Z}^{i;k,t,x}\mathrm{d}\mathrm{B}_{r}-\displaystyle\int^{T}_{s}
\displaystyle\int_{\mathrm{E}}\overline{U}^{i;k,t,x}_{r}(e)\tilde{\mu}(\mathrm{d}r,\mathrm{d}e)\\
\quad\quad\quad+\displaystyle\int^{T}_{s}h^{(i)}\left(r,X^{k,t,x}_{r},\overline{Y}^{i;k,t,x}_{r},\overline{Z}^{i;k,t,x}_{r},
\right.\\

\left.
\qquad\qquad\qquad\qquad\qquad\displaystyle\int_{\mathrm{E}}\gamma^{i}(t,X^{k,t,x_{k}}_{r},e)(\overline{u}_{i}(s,X^{k,t,x_{k}}_{s-}+\beta(s,X^{k,t,x_{k}}_{s-},e))-\overline{u}_{i}(s,X^{k,t,x_{k}}_{s-}))\,\lambda(de)\right)dr;\\
(iii)~\overline{Y}^{i;k,t,x_{k}}_{T}=g(X^{k,t,x_{k}}_{T}).
\end{array}
\right.
\end{equation}
and
\begin{equation}\label{5.73}
\left
\{\begin{array}{ll}
(i)~\vec{\overline{Y}}^{k,t,x_{k}}:=(\overline{Y}^{i;k,t,x_{k}})_{i=1,m}\in\mathcal{S}^{2}(\mathbb{R}^{m}),~\overline{Z}^{k,t,x_{k}}:=(\overline{Z}^{i;k,t,x_{k}})_{i=1,m}\in\mathbb{H}^{2}(\mathbb{R}^{m\times d}),\\\qquad\qquad\overline{U}^{k,t,x_{k}}:=(\overline{U}^{i;k,t,x_{k}})_{i=1,m}\in\mathbb{H}^{2}(\mathbb{L}^{2}_{m}(\lambda));\\
(ii)~\overline{Y}^{i;k,t,x_{k}}_{s}= g^{i}(X^{k,t,x_{k}}_{T})-\displaystyle\int^{T}_{s}\overline{Z}^{i;k,t,x_{k}}\mathrm{d}\mathrm{B}_{r}-\displaystyle\int^{T}_{s}
\displaystyle\int_{\mathrm{E}}\overline{U}^{i;k,t,x_{k}}_{r}(e)\tilde{\mu}(\mathrm{d}r,\mathrm{d}e)\\
\quad\quad\quad+\displaystyle\int^{T}_{s}h^{(i)}\left(r,X^{k,t,x_{k}}_{r},\overline{Y}^{i;k,t,x_{k}}_{r},\overline{Z}^{i;k,t,x_{k}}_{r},
\right.\\

\left.
\qquad\qquad\qquad\qquad\qquad\displaystyle\int_{\mathrm{E}}\gamma^{i}(t,X^{k,t,x_{k}}_{r},e)(\overline{u}_{i}(s,X^{k,t,x_{k}}_{s-}+\beta(s,X^{k,t,x_{k}}_{s-},e))-\overline{u}_{i}(s,X^{k,t,x_{k}}_{s-}))\,\lambda(de)\right)dr;\\
(iii)~\overline{Y}^{i;k,t,x_{k}}_{T}=g(X^{k,t,x_{k}}_{T}). 
\end{array}
\right.
\end{equation}
By proof of step $2$ of proposition \ref{pro4.3}, $(\overline{Y}^{i;k,t,x},\overline{Z}^{i;k,t,x},\overline{U}^{i;k,t,x}1_{\{|e|\geq \frac{1}{k}\}})_{k}$ converge to $(\overline{Y}^{i;t,x}, \overline{Z}^{i;t,x},\overline{U}^{i;t,x})$ in $\mathcal{S}^{2}(\mathbb{R})\times\mathbb{H}^{2}(\mathbb{R}^{\kappa\times d})\times\mathbb{H}^{2}(\mathbb{L}^{2}(\lambda)) $.\\
Let $((v^{k}_{i=1,m}))_{k\geq 1}$ be the sequence of continuous deterministics functions such that for any $t\leq T$ and $s\in[t,T]$,\\
$$\overline{Y}^{i;k,t,x}_{s}=\overline{v}^{k}_{i}(s,^{k}X^{t,x}_{s})~\text{and } \overline{Y}^{i;k,t,x_{k}}_{s}=\overline{v}^{k}_{i}(s,^{k}X^{t,x_{k}}_{s})~~\forall i=1,\ldots,m.$$
Such that we have respectively by proof of proposition \ref{pro4.3} in step $1$ and step $2$:\\
$(i)~\overline{U}^{i;k,t,x}_{s}(e)=(v^{i}(s,^{k}X^{t,x}_{s-}+\beta(s,^{k}X^{t,x}_{s-},e))-v^{i}(s,^{k}X^{t,x}_{s-}))$, $ds\otimes d\mathbb{P}\otimes d\lambda_{k}$-a.e on $[t,T]\times\Omega\times E$;\\
$(ii)~\text{the sequence } ((v^{k}_{i=1,m}))_{k\geq 1}$ converges to $v^{i}(t,x)$
by using (\ref{4.59}).\\
So that $x_{k}\longrightarrow_{k} x$ we take the following estimation which is obtaining by Ito's formula and by the properties of $h^{(i)}$.
\begin{eqnarray}\label{5.74}
& {}{} &\mathbb{E}\left[\left|\vec{Y}^{k,t,x_{k}}_{s}-Y^{k,t,x}_{s}\right|^{2}+\displaystyle\int^{T}_{0}\left\lbrace\left|Z^{k,t,x_{k}}_{s}-Z^{k,t,x}_{s}\right|^{2}+\displaystyle\int_{E}\left|U^{k,t,x_{k}}_{s}-U^{k,t,x}_{s}\right|^{2}\,\lambda_{k}(de)\right\rbrace\,ds\right]\nonumber\\
& {}{} &\leq\mathbb{E}\left[\left|g(^{k}X^{t,x_{k}}_{T})-g(^{k}X^{t,x}_{T})\right|^{2}+2\displaystyle\int^{T}_{s}<\left(\vec{Y}^{k,t,x_{k}}_{r}-Y^{k,t,x}_{r}\right),^{k}\Delta h^{(i)}(r)>\,dr\right]\nonumber
\end{eqnarray}

Using the same arguments as in proof (\ref{4.51}), it follows that for $s=t$; $\forall i=1,\ldots,m$,\\
$$v^{k}_{i}(t,x_{k})\longrightarrow_{k}v^{k}_{i}(t,x).$$
Therefore by (i)-(ii) we have, for any $i=1,\ldots,m$,\\
\begin{equation}\label{5.75}
\overline{U}^{i;t,x}_{s}(e)=(v^{i}(s,X^{t,x}_{s-}+\beta(s,X^{t,x}_{s-},e))-v^{i}(s,X^{t,x}_{s-}))\quad ds\otimes d\mathbb{P}\otimes d\lambda-\text{a.e. in  }[t,T]\times\Omega\times E, \quad\forall i\in\{1,\ldots, m\}.
\end{equation}
By this result we can replace $(\overline{u}_{i}(s,X^{t,x}_{s-}+\beta(s,X^{t,x}_{s-},e))-\overline{u}_{i}(s,X^{t,x}_{s-}))$ by $\overline{U}^{i;t,x}_{s}(e)$ in (\ref{5.71}), we deduce that $(\overline{Y}^{t,x},\overline{Z}^{t,x},\overline{U}^{t,x})$ verifies: $\forall i\in\{1,\ldots,m\}$ 
\begin{equation}\label{5.76}
\left
\{\begin{array}{ll}
(i)~\vec{\overline{Y}}^{t,x}:=(\overline{Y}^{i;t,x})_{i=1,m}\in\mathcal{S}^{2}(\mathbb{R}^{m}),~\overline{Z}^{t,x}:=(\overline{Z}^{i;t,x})_{i=1,m}\in\mathbb{H}^{2}(\mathbb{R}^{m\times d}),~\overline{U}^{t,x}:=(\overline{U}^{i;t,x})_{i=1,m}\in\mathbb{H}^{2}(\mathbb{L}^{2}_{m}(\lambda));\\
(ii)~\overline{Y}^{i;t,x}_{s}= g^{i}(X^{t,x}_{T})-\displaystyle\int^{T}_{s}\overline{Z}^{i;t,x}_{r}\,\mathrm{d}
\mathrm{B}_{r}-\displaystyle\int^{T}_{s}
\displaystyle\int_{\mathrm{E}}\overline{U}^{i;t,x}_{r}(e)\,\tilde{\mu}(\mathrm{d}r,\mathrm{d}e).\\
\quad\quad\quad+\displaystyle\int^{T}_{s}h^{(i)}(r,X^{t,x}_{r},\overline{Y}^{i;t,x}_{r},\overline{Z}^{i;t,x}_{r},\displaystyle\int_{\mathrm{E}}\gamma^{i}(r,X^{t,x}_{r},e)\overline{U}^{i;t,x}_{r}\,\lambda(de))dr;\\
(iii)~\overline{Y}^{i;t,x}_{T}=g^{i}(X^{t,x}_{T}).
\end{array}
\right.
\end{equation}
It follows that $$\forall i\in\{1,\ldots,m\},\quad \overline{Y}^{i;t,x}={Y}^{i;t,x}.$$
With the uniqueness of solution (\ref{5.68}), we have $u^{i}=\overline{u}^{i}=v^{i}$ which means that the solution of (\ref{eq1}) in the sense of Definition \ref{def5.1} is unique inside the class $\mathcal{M}$. 
\end{proof}
\\
In this paper we have shown the existence and uniqueness of viscosity solution through BSDE by weakening the condition on the generator used in \cite{hama}.\\
The question that arose now is to have the existence and uniqueness of viscosity solution through the RBSDEs.
\newpage
\section*{Appendix}
\subsection*{Barles and al.'s definition for viscosity solution of IPDE (\ref{eq1})}
\begin{definition}\label{def5.3}
We say that a family of deterministics functions $u=(u^{i})_{i=1,m}$ which is continuous $\forall i\in\{1,\ldots,m\}$, is a viscosity sub-solution (resp. super-solution) of the IPDE (\ref{eq1}) if:\\
$(i)\quad \forall x\in\mathbb{R}^{k}$, $u^{i}(x,T)\leq g^{i}(x)$ (resp. $u^{i}(x,T)\geq g^{i}(x)$);\\
$(ii)\quad\text{For any } (t,x)\in[0,T]\times\mathbb{R}^{k}$ and any function $\phi$ of class $C^{1,2}([0,T]\times\mathbb{R}^{k})$ such that $(t,x)$ is a global maximum point of $u^{i}-\phi$ (resp. global minimum point of $u^{i}-\phi$) and $(u^{i}-\phi)(t,x)=0$,  one has
\begin{equation*}
\min\left\lbrace u^{i}(t,x)-\ell(t,x);-\partial_{t}\phi(t,x)-\mathcal{L}^{X}\phi(t,x)-h^{i}(t,x,(u^{j}(t,x))_{j=1,m},\sigma^{\top}(t,x))D_{x}\phi(t,x),B_{i}\phi(t,x))\right\rbrace\leq 0 
\end{equation*}
$\left(resp.
\right.$
\begin{equation*}
\left.
\min\left\lbrace u^{i}(t,x)-\ell(t,x);-\partial_{t}\phi(t,x)-\mathcal{L}^{X}\phi(t,x)-h^{i}(t,x,(u^{j}(t,x))_{j=1,m},\sigma^{\top}(t,x))D_{x}\phi(t,x),B_{i}\phi(t,x)(t,x))\right\rbrace\scriptstyle\geq 0\right).
\end{equation*}
The family $u=(u^{i})_{i=1,m}$ is a viscosity solution of (\ref{eq1}) if it is both a viscosity sub-solution and viscosity super-solution.\\
Note that $\mathcal{L}^{X}\phi(t,x)=b(t,x)^{\top}\mathrm{D}_{x}\phi(t,x)+\frac{1}{2}\mathrm{Tr}(\sigma\sigma^{\top}(t,x)\mathrm{D}^{2}_{xx}\phi(t,x))+\mathrm{K}\phi(t,x)$;\\
where $\mathrm{K}\phi(t,x)=\displaystyle\int_{\mathrm{E}}(\phi(t,x+\beta(t,x,e))-\phi(t,x)-\beta(t,x,e)^{\top}\mathrm{D}_{x}\phi(t,x))\lambda(de)$.
\end{definition}
\subsection*{Another Mao condition}
In this paper it was mainly a question of the p-order Mao condition, it is necessary to know that there exists another condition of Mao which is mainly used in the case of monotony of the generator for apply the comparison theorem at the viscosity solution.
\begin{definition}
$f$ satisfies the $p$-order one-sided Mao condition in $x$ i.e., there  exists a nondecreasing, concave function $\rho(\cdot):\mathbb{R}^{+}\mapsto\mathbb{R}^{+}$ with $\rho(0)=0$, $\rho(u)>0$, for $u>0$ and $\displaystyle\int_{0^{+}}\frac{du}{\rho(u)}=+\infty$, such that $d\mathbb{P}\times dt-$a.e., $\forall x, x^{'}\in\mathbb{R}^{k}\textrm{ and }\forall p\geq 2$,
$$<\frac{x-x'}{|x-x'|}\mathds{1}_{|x-x'|\neq 0},f(t,x,y,z,q)-f(t,x^{'},y,z,q)>\leq\rho^{\frac{1}{p}}(|x-x^{'}|^{p}).$$
\end{definition}
\begin{remark}
Applying  Cauchy-Schwartz inequality for the $p$-order one-sided Mao condition, we deduce the $p$-order Mao condition.
\end{remark}

\end{document}